\documentclass{amsart}
\title[Full groups of minimal homeomorphisms and Baire category methods]{Full groups of minimal homeomorphisms and Baire category methods}
\author{Tom\'as Ibarluc\'ia}
\author{Julien Melleray}
\address{Universit\'e de Lyon \\
CNRS UMR 5208 \\
Universit\'e Lyon 1 \\
Institut Camille Jordan \\
43 blvd. du 11 novembre 1918\\
F-69622 Villeurbanne Cedex \\
France}

\usepackage{amssymb}
\usepackage[initials,shortalphabetic]{amsrefs}
\usepackage[all]{xy}
\usepackage[small, nohug,heads=vee]{diagrams}
\diagramstyle[labelstyle=\scriptstyle]
\usepackage{mathpazo}

\usepackage{my-macros}
\numberwithin{equation}{section}

\begin{document}

\begin{abstract}
We study full groups of minimal actions of countable groups by homeomorphisms on a Cantor space $X$, showing that these groups do not admit a compatible Polish group topology and, in the case of $\Z$-actions, are coanalytic non-Borel inside $\Homeo(X)$.  We point out that the full group of a minimal homeomorphism is topologically simple. We also study some properties of the closure of the full group of a minimal homeomorphism inside $\Homeo(X)$.
\end{abstract}
\maketitle
%%%%%%%%%%%%%%%%%%%%%%%%%%%%%%%%%%%%%%%%%%%%%%%%%%%%%%%%%%%%%%%%%%%%%%%%%%%%%%%%%%%%%%%%%%%%%%%%%%%%%%%%%%%%%%%%%%%%%%
%%%%%%%%%%%%%%%%%%%%%%%%%%%%%%%%%%%%%%%%%%%%%%%%%%%%%%%%%%%%%%%%%%%%%%%%%%%%%%%%%%%%%%%%%%%%%%%%%%%%%%%%%%%%%%%%%%%%%%
\section{Introduction}
When studying a mathematical structure, one is often led
%common practice: isn't it a bit weak as a motivation? Julien: do you have a suggestion? I always have a hard time finding a good first sentence... Tomás: an alternative, but probably not any better (I will keep on thinking): "A mathematical structure might be behold through its automorphism group, frequently with a hope of keeping the relevant information intact. A particularly striking example in ergodic theory is provided by a theorem of Dye:..."Julien: I tried to make a change... The motivation is not stronger than before, I'm afraid.
 to consider the properties of its automorphism group, and then it is tempting to ask to what extent the group characterizes the structure. A particularly striking example is provided by a theorem of Dye (\cite{Dye1959}, \cite{Dye1963}) in ergodic theory: assume that two countable groups $\Gamma, \Delta$ act on the unit interval $[0,1]$ by measure-preserving automorphisms, without any non-trivial invariant sets (i.e.~the actions are \emph{ergodic}), and consider the groups $[\Gamma]$ (resp. $[\Delta]$) made up of all measurable bijections that map each $\Gamma$-orbit (resp. $\Delta$-orbit) onto itself. Then the groups $[\Gamma]$ and $[\Delta]$ are isomorphic if, and only if, there exists a measure-preserving bijection of $[0,1]$ which maps $\Gamma$-orbits onto $\Delta$-orbits. One then says that the relations are orbit equivalent; $[\Gamma]$ is called the \emph{full group} of the action. Using this language, Dye's theorem says that the full group of an ergodic action of a countable group on a standard probability space completely remembers the associated equivalence relation up to orbit equivalence.

This result was the motivation for an intensive study of full groups in ergodic theory, for which we point to \cite{Kechris2010} as a general reference. More recently, it came to light, initially via the work of Giordano--Putnam--Skau, that a similar phenomenon takes place in \emph{topological dynamics}. In that context, one still considers actions of countable groups, replacing probability-measure-preserving actions with actions by homeomorphisms of a Cantor space. The two settings are related: for instance, when $\Gamma$ is a countable group, one could consider the Bernoulli shift action of $\Gamma$ on $\{0,1\}^\Gamma$ as a measure-preserving action (say, for the $(1/2,1/2)$-Bernoulli measure) or as an action by homeomorphisms. 
As in the measure-theoretic setting, one can define the full group of an action of a countable group $\Gamma$ on a Cantor space $X$: this time, it is made up of all homeomorphisms of $X$ which map $\Gamma$-orbits onto themselves. The counterpart of ergodicity here is \emph{minimality}, i.e.~the assumption that all the orbits of the action are dense; the analog of Dye's theorem for minimal group actions  was proved by Giordano--Putnam--Skau \cite{Giordano1999}. 

The measure-theoretic and topological settings may appear, at first glance, to be very similar; however, there are deep differences, for instance all ergodic group actions of countable amenable groups are orbit equivalent (Connes--Feldman--Weiss \cite{Connes1981}) while there exists a continuum of pairwise 
%Julien: I checked, and this is the proper reference. The same result for $\Z$ instead of general amenable groups is due to Dye.
non-orbit equivalent actions of $\Z$ by homeomorphisms of a Cantor space. Still, it is interesting to investigate properties of full groups in topological dynamics, which has been done by several authors over the last twenty years or so. 

In both contexts discussed above, it is natural to consider the full group of an action as a topological group, the topology being induced by the topology of the ambient Polish group (measure-preserving bijections of $[0,1]$ in one case, homeomorphisms of the Cantor space in the other). The usefulness of this approach is however limited by the fact that the full group of an ergodic group action, or a minimal group action, is not closed in the ambient group; in the first case the full group is dense, in the second it seems that the closure is currently only understood for actions of $\Z^d$.

It then comes as a blessing that, in the measure-theoretic context, one can endow the full group with a stronger topology which turns it into a Polish group: the \emph{uniform topology}, induced by the distance given by $d(g,h)= \mu(\{x \colon g(x) \ne h(x)\})$. This paper grew out of the following question: can one do the same thing in the topological context? It is interesting to note that, shortly after the publication of \cite{Giordano1999}, Bezyglyi and Kwiatkowski \cite{Bezuglyi2002} introduced an analog of the uniform topology in the context of topological dynamics, which is however far from being as nice as the uniform topology of ergodic theory. This provides further motivation for trying to understand whether a nice group topology exists at all.

\begin{theorem*}
Let $\Gamma$ be a countable group acting minimally by homeomorphisms on a Cantor space $X$. Then any Hausdorff, Baire group topology on $\left[\Gamma\right]$ must extend the topology of pointwise convergence for the discrete topology on $X$.
Consequently, there is no second countable, Hausdorff, Baire group topology on $\left[\Gamma\right]$.
\end{theorem*}

This is bad news, but certainly not surprising ---if a Polish group topology existed for that group, it would have been considered a long time ago. In the same spirit, one can then wonder about the complexity of full groups inside the ambient automorphism group; in ergodic theory, full groups are always fairly tame, in the sense that they can be written as countable intersections of countable unions of closed sets \cite{Wei2005}. Yet again, the situation turns out to be more dire in topological dynamics.

\begin{theorem*}
The full group of a minimal homeomorphism of a Cantor space $X$ is a coanalytic non-Borel subset of $\Homeo(X)$.
\end{theorem*}

This led us to study the \emph{closure} of a full group inside the homeomorphism group; this is a Polish group, and is also a complete invariant for orbit equivalence if one is willing to restrict one's attention to actions of $\Z$, which we do in the last sections of this article. It follows from a theorem of Glasner--Weiss \cite{Glasner1995a} that the closure of the full group of a minimal homeomorphism $\varphi$ coincides with the group of homeomorphisms which preserve all $\varphi$-invariant measures. Using work of Bezuglyi--Medynets and Grigorchuk--Medynets, we obtain the following result \footnote{We originally proved this only for uniquely ergodic homeomorphisms; we thank K. Medynets for explaining how to make the argument work in general.}.

\begin{theorem*}
The closure of the full group of a minimal homeomorphism of the Cantor space is topologically simple (hence, the full group itself is also topologically simple).
\end{theorem*}

It is an open problem whether the full group of a minimal homeomorphism is simple.

In the case of uniquely ergodic homeomorphisms, we also provide a criterion for the existence of dense conjugacy classes in the closure of the full group (in terms of the values taken by the unique invariant measure on clopen sets), and use a Fra\"iss\'e theoretic approach to recover a result of Akin which describes a class of uniquely ergodic homeomorphisms with the property that the closure of their full group admits a comeager conjugacy class.

\medskip

\emph{Acknowledgements.} The second author's interest in the subject was kindled by two meetings organized by D. Gaboriau and funded by the ANR Network AGORA, and by lectures given by T. Giordano and K. Medynets at these meetings. We are grateful to I. Ben Yaacov, D. Gaboriau and B. Weiss for useful discussions and suggestions. We are indebted to K. Medynets for valuable comments made after reading an earlier version of this article.

The second author's research was partially funded by the ANR network AGORA, NT09-461407 and ANR project GRUPOLOCO, ANR-11-JS01-008.

%Never done this, but I suppose I should put something like "The first author's research was (partially?) funded by (a Contrat Doctoral from?) Université Claude Bernard Lyon 1."? Julien: not sure; I think that this is implied by your institutional address being at the Université Lyon 1 (the fundings that one usually mentioned are special grants, etc. However, does your funding depend in any way from MILYON? If so it would be worth mentioning, I think.
%Tomás: OK. I don't know what MILYON is, so I guess my funding does not depend from it. On the other hand, maybe I should thank Itaï, with whom I discussed the Hrushovski extension property in the beggining. I am also grateful to the second author for introducing and sharing the subject to me so generously, and for teaching me so much.

%%%%%%%%%%%%%%%%%%%%%%%%%%%%%%%%%%%%%%%%%%%%%%%%%%%%%%%%%%%%%%%%%%%%%%%%%%%%%%%%%%%%%%%%%%%%%%%%%%%%%%%%%%%%%%%%%%%%%%
%%%%%%%%%%%%%%%%%%%%%%%%%%%%%%%%%%%%%%%%%%%%%%%%%%%%%%%%%%%%%%%%%%%%%%%%%%%%%%%%%%%%%%%%%%%%%%%%%%%%%%%%%%%%%%%%%%%%%%

\section{Background and terminology} 

We now go over some background material and discuss in more detail some facts that were mentioned briefly in the introduction.

Recall that a \emph{Cantor space} is a nonempty, zero-dimensional, perfect compact metrizable space; any two Cantor spaces are homeomorphic. Given a Cantor space $X$, we denote by $\Clop(X)$ the Boolean algebra of all clopen subsets of $X$, and by $\Homeo(X)$ the group of homeomorphisms of $X$. This group can be endowed with the topology whose basic open sets are of the form 
$$\{g \in \Homeo(X) \colon \forall i \in \{1,\ldots,n\}\ g(U_i)=V_i \} \ ,  $$
where $n$ is an integer and $U_i,V_i$ are clopen subsets of $X$. This turns $\Homeo(X)$ into a \emph{topological group}, namely the group operations $(g,h) \mapsto gh$ and $g \mapsto g^{-1}$ are continuous with respect to this topology. 

\begin{defn}
A \emph{Polish group} is a topological group whose topology is induced by a complete, separable metric.
\end{defn}

Picking a compatible distance $d$ on $X$, one can check that the topology defined above on $\Homeo(X)$ is a Polish group topology, a compatible complete distance being given by 
$$d(g,h)= \max_{x \in X}d(g(x),h(x)) + \max_{y \in X} d(g^{-1}(y),h^{-1}(y)) \ . $$
It might be a bit surprising at first that the two topologies we defined coincide; actually, this is a hint of a more general phenomenon: the unique second--countable group 
topologies on $\Homeo(X)$ are the coarse topology and the Polish group topology we defined above (this follows from results of \cite{Anderson1958}, \cite{Gamarnik1991} and \cite{Rosendal2007}).

Polish groups form a fairly general class of groups, yet the combination of separability and the use of Baire category methods make them relatively tame. The fact that the Baire category theorem holds in Polish groups is particularly important; we recall that, whenever $G$ is a topological group for which the Baire category theorem holds, $H$ is a separable topological group and $\varphi \colon G \to H$ is a Borel homomorphism, then $\varphi$ must actually be continuous (see e.g.~\cite{Kechris1995}*{Theorem~9.10}). 

\begin{defn}
Let $\Gamma$ be a countable group acting by homeomorphisms on a Cantor space $X$. We denote by $R_{\Gamma}$ the associated equivalence relation and define its \emph{full group} as the group of all homeomorphisms of $X$ which preserve each $\Gamma$-orbit; in symbols, 
$$\left[R_{\Gamma} \right]= \{g \in \Homeo(X) \colon \forall x \in X \exists \gamma \in \Gamma \ g(x)= \gamma \cdot x\} \ . $$
\end{defn}

As is the case in ergodic theory, the full group of an action of a countable group by homeomorphisms of a Cantor space $X$ completely remembers the associated equivalence relation, a fact made precise by the following definition and theorem.

\begin{defn}
Let $\Gamma_1,\Gamma_2$ be two countable groups acting by homeomorphisms on a Cantor space $X$, and let $R_{\Gamma_1}, R_{\Gamma_2}$ be the associated equivalence relations. We say that $R_{\Gamma_1}$ and $R_{\Gamma_2}$ are \emph{orbit-equivalent} if there is a homeomorphism $g$ of $X$ such that 
$$\forall x,y \in X \ \left( x R_{\Gamma_1} y \right) \Leftrightarrow \left( g(x) R_{\Gamma_2} g(y) \right) \ . $$
\end{defn}

\begin{theorem}[\cite{Giordano1999}, \cite{Medynets2011}]
Let $\Gamma_1$, $\Gamma_2$ be countable groups acting by homeomorphisms on a Cantor space $X$; assume that all orbits for both actions have cardinality at least $3$, and for any nonempty $U \in \Clop(X)$ and $i=1,2$ there exists $x \in U$ such that $\Gamma_i \cdot x$ intersects $U$ in at least two points. 

Denote by $R_{\Gamma_1}$ and $R_{\Gamma_2}$ the associated equivalence relations, and suppose that there exists an isomorphism $\Phi$ from $\left[R_{\Gamma_1} \right]$ to $\left[R_{\Gamma_2} \right]$. Then there must exist $g \in \Homeo(X)$ such that $\Phi(h)= g h g^{-1}$ for all $h \in \left[R_{\Gamma_1} \right]$.

Consequently, $\left[R_{\Gamma_1}\right ]$ and $\left[R_{\Gamma_1}\right ]$ are isomorphic if, and only if, $R_{\Gamma_1}$ and $R_{\Gamma_2}$ are orbit equivalent.

\end{theorem}

The above result was first proved by Giordano--Putnam--Skau \cite{Giordano1999} for \emph{minimal} actions, which we define now, and then extended by Medynets \cite{Medynets2011}.

\begin{defn}
Let $\Gamma$ be a countable group acting by homeomorphisms on a Cantor space. We say that the action is \emph{minimal} if every point has a dense orbit.
\end{defn}

Minimal actions are particularly well-studied when $\Gamma=\Z$. In that case the action is simply induced by one homeomorphism $\varphi$; accordingly, we will use the notation $\left[ \varphi \right]$ to denote the full group of the associated equivalence relation. Similarly, when the $\Z$-action associated to a homeomorphism $\varphi$ is minimal we simply say that $\varphi$ is minimal. In the case of minimal actions of $\Z$, a particular subgroup plays an important role and is well understood, mainly thanks to work of Matui.

\begin{defn}
Let $\varphi$ be a homeomorphism of a Cantor space $X$. Its \emph{topological full group} $\left[ [ \varphi ]\right]$ is the set of elements $g \in \Homeo(X)$ for which there is a finite clopen partition $U_1,\ldots,U_n$ of $X$ such that on each $U_i$ $g$ coincides with some power of $\varphi$. 
\end{defn}

The topological full group $\left[ [ \varphi ]\right]$ is a countable subgroup of $\left[\varphi \right]$; note that, if all orbits of $\varphi$ are infinite, then for any element $g$ in
$\left[\varphi \right]$ there exists a unique $n_x \in \Z$ such that $g(x) = \varphi^{n_x}(x)$; the group $\left[ [ \varphi ]\right]$ is simply made up of all $g$ for which the associated cocycle $x \mapsto n_x$ is continuous.
Another equivalent (though apparently weaker) definition is that $\left[ [ \varphi ]\right]$ is exactly the set of elements of $\left[\varphi \right]$ for which the map $x \mapsto n_x$ has a finite range. Indeed, each set of the form $\{x \in X \colon g(x)=\varphi^n(x)\}$ is closed, and these sets cover $X$ if $g$ belongs to $\left[\varphi\right]$, so if there are only finitely many nonempty such sets then they are clopen and $g$ belongs to $\left[ [ \varphi ]\right]$.

The next result elucidates the action of the full group of a minimal homeomorphism on the algebra of clopen sets. Before stating it we set some notation for the sequel.
\medskip

\emph{Notation.} Let $X$ be a Cantor space and $\varphi$ a homeomorphism of $X$. A Borel probability measure $\mu$ on $X$ is $\varphi$-\emph{invariant} if $\mu(A)=\mu(\varphi^{-1}(A))$ for any Borel subset $A \subseteq X$ (if this equality holds for clopen sets then it must hold for all Borel sets). Given a homeomorphism $\varphi$, we denote by $\mathcal M_{\varphi}$ its set of invariant probability measures, which is a nonempty compact, convex subset of the space of all probability measures on $X$. 
\medskip

\begin{theorem}[\cite{Glasner1995a}*{Lemma 2.5 and Proposition 2.6}]\label{GlasnerWeiss}
Let $\varphi$ be a minimal homeomorphism of a Cantor space $X$, and $A,B$ be clopen subsets of $X$. Then the following facts hold:
\begin{itemize}
\item If $\mu(A)< \mu(B) $ for all $\mu \in \mathcal M_{\varphi}$ then there exists $ g \in \left[ [\varphi ]\right]$ such that $g(A) \subset B $.

\item $\left( \forall \mu \in \mathcal M_{\varphi} \ \mu(A)= \mu(B) \right) \Leftrightarrow \left( \exists g \in \left[ \varphi \right] g(A) = B \right)$.
\end{itemize}

\end{theorem}

\begin{remark}
In both cases, one can add the assumption that $g^2=1$ to the right-hand statement. It is useful that in the first statement above, one can find $g$ in the topological full group and not merely in the full group; this is not always possible when it comes to the second statement.
\end{remark}

%%%%%%%%%%%%%%%%%%%%%%%%%%%%%%%%%%%%%%%%%%%%%%%%%%%%%%%%%%%%%%%%%%%%%%%%%%%%%%%%%%%%%%%%%%%%%%%%%%%%%%%%%%%%%%%%%%%%%%
%%%%%%%%%%%%%%%%%%%%%%%%%%%%%%%%%%%%%%%%%%%%%%%%%%%%%%%%%%%%%%%%%%%%%%%%%%%%%%%%%%%%%%%%%%%%%%%%%%%%%%%%%%%%%%%%%%%%%%

\section{Topologies on full groups.}\label{topology}
Throughout this section we let $\Gamma$ denote a countable group acting on a Cantor space $X$, and $R$ denote the associated equivalence relation. We make the following assumption on the action of $\Gamma$: given any nonempty open $U \subseteq X$, there exist $x \ne y \in U$ and $\gamma \in \Gamma$ such that $\gamma \cdot x = y$. Equivalently, there exists no nonempty open subset $U \subseteq X$ such that the restriction of $R$ to $U$ is trivial.  

Note that this assumption implies that, given any nonempty open $U$, there exists $g \in \left[ R \right]$ and a clopen $V$ such that $g^2=1$, $g(V)$ and $V$ are nonempty disjoint clopen subsets of $U$ and $g$ coincides with the identity outside of $V \cup g(V)$. We now point out a further consequence.

\begin{lemma} \label{InfiniteOrbits}
The set $\Omega =\{x \in X \colon x \text{ is an accumulation point of } \Gamma \cdot x\}$ is dense $G_{\delta}$ in~$X$.
\end{lemma}

\begin{proof}
Fix a compatible metric $d$ on $X$. Recall that $x$ is an accumulation point of $\Gamma \cdot x$ if, and only if, the following condition holds:
$$\forall \varepsilon >0 \exists \gamma \in \Gamma \ \left(\gamma \cdot x \ne x \text{ and } d(\gamma \cdot x, x) < \varepsilon \right) $$
The condition between brackets is open, showing that $\Omega$ is indeed $G_{\delta}$. By the Baire category theorem, to check that $\Omega$ is dense we only need to prove that for any $\varepsilon >0$ the set $\{x \colon \exists \gamma \in \Gamma \ \gamma \cdot x \ne x \text{ and } d(\gamma \cdot x, x) < \varepsilon \}$ is dense, and this follows immediately from our assumption on the action.
\end{proof}

For the next two lemmas and proposition, we let $\tau$ denote a group topology on $\left[ R \right]$ which is Hausdorff and such that $(\left[ R \right],\tau)$ is a Baire space. 

\begin{lemma}\label{AlgCar1}
For any nonempty clopen subset $U$ of $X$, the set $\Omega_U =\{ g \in \left[ R \right] \colon g_{\upharpoonright U}=id_{\upharpoonright U}\}$ is $\tau$-closed.
\end{lemma}

\begin{proof}
We claim that $g \in \left[ R \right]$ coincides with the identity on $U$ if, and only if, $gh=hg$ for any $h \in \left[ R \right]$ whose support is contained in $U$. Each set $\{g \colon gh=hg\}$ is closed since $\tau$ is a Hausdorff group topology, hence if we prove this claim we can conclude that $\Omega_U$ is an intersection of closed subsets of $\left[ R \right]$ so $\Omega_U$ is closed.

Now to the proof of the claim: one inclusion is obvious; to see the converse, assume that there exists $x \in U$ such that $x \ne g(x)$. This gives us a clopen subset $W$ of $U$ such that $W$ and $g(W)$ are disjoint. By assumption, there exists a clopen subset $V$ of $W$ and an involution $h \in \left[ R \right]$ with support contained in $W$ (hence in $U$) and such that $V$ and $h(V)$ are disjoint subsets of $W$. Then $hg(V)=g(V)$ is disjoint from $gh(V)$, showing that $g$ and $h$ do not commute.
\end{proof}

\begin{lemma}\label{AlgCar2}
For any clopen subset $U$ of $X$, the set $\Sigma_U =\{g \in \left[ R \right] \colon g(U)=U\}$ is $\tau$-closed.
\end{lemma}
\begin{proof}
We may assume that $U$ is nonempty; also, since $\tau$ is a group topology and $(g(U)=U) \Leftrightarrow (g(U) \subseteq U \text{ and } g^{-1}(U) \subseteq U)$, we only need to show that 
$\{g \in \left[ R \right] \colon g(U) \subseteq U \}= \Sigma_U'$ is closed in $\left[ R \right]$. To that end, one can use the same strategy as above:  this time, we claim that $g \in \Sigma_U'$ if and only if, for any $h$ which coincides with the identity on $U$, $g^{-1}hg$ coincides with the identity on $U$. Proving this will show that $\Sigma_U'$ is an intersection of closed sets (by Lemma \ref{AlgCar1}), which gives the result.

Again, one inclusion is obvious; to see the converse, we assume that $g(U)$ is not contained in $U$. Then there exists a nonempty clopen subset $W$ of $U$ such that $g(W) \cap U= \emptyset$. 
One can find a nontrivial involution $h$ with support in $g(W)$. This gives us a nonempty clopen $V \subseteq U$ such that $hg(V)$ and $g(V)$ are disjoint, hence $g^{-1}hg$ does not coincide with the identity on $U$. 
\end{proof}

\begin{prop} \label{AlgCar3}
The set $\{ g \in \left[ R \right] \colon g(x)=x\}$ is $\tau$-clopen for all $x \in X$.

\end{prop}
\begin{proof}
The result of Lemma \ref{AlgCar2} shows that the natural inclusion map from $(\left[ R \right],\tau)$ to $\Homeo(X)$ is Borel. Since $\tau$ is assumed to be Baire (this is the first time we are using that assumption) and $\Homeo(X)$ is separable, the inclusion map must be continuous, showing that each set $\{ g \in \left[ R \right] \colon g(x)=x\}$ is $\tau$-closed. Now, fix $x \in X$ and let $H$ denote the permutation group of the countable set $\Gamma \cdot x$, endowed with its permutation group topology. Since $\{g \in \left[ R \right] \colon g(\gamma_1 \cdot x)= \gamma_2 \cdot x\}$ is closed for all $\gamma_1,\gamma_2 \in \Gamma$, we see that the natural homomorphism from $(\left[ R \right],\tau)$ to $H$ (given by $g \mapsto (\gamma \cdot x \mapsto g(\gamma \cdot x))$ is Borel. Thus this homomorphism must be continuous, and $\{ g \in \left[ R \right] \colon g(x)=x\}$ is $\tau$-clopen.
%no more magic wands? Julien: I liked the magic wand too, but thought it might be a bit much. We may try to put it back, if you want (at worst the referee will ask us to remove it - I think that's likely to happen unfortunately). Tomás: there's no Journal of Mathemagics? Julien: there should be!
\end{proof}

Let us sum up what we just proved.

\begin{theorem}
Let $\Gamma$ be a countable group acting by homeomorphisms on a Cantor space $X$. Assume that the restriction of the associated equivalence relation to a nonempty open subset of $X$ is never trivial. 

Then any Hausdorff, Baire group topology on $\left[ R \right]$ must extend the topology of pointwise convergence for the discrete topology on $X$.
Consequently, there is no second countable, Hausdorff, Baire group topology on $\left[ R \right]$.
\end{theorem}
\begin{proof}
The statement in the first sentence corresponds exactly to the result of Proposition \ref{AlgCar3}. To see why the second statement holds, let us proceed by contradiction and assume that there exists a second countable, Hausdorff, Baire group topology on $\left[ R \right]$.

 We recall the result of Lemma \ref{InfiniteOrbits}: the set 
$\Omega$ made up of all $x$ such that $x$ is an accumulation point of $\Gamma \cdot x$ is dense $G_{\delta}$ in $X$, thus in particular uncountable.
Assuming $\tau$ is second countable, the Lindel\"off property implies that there exists a sequence $(x_i)_{i< \omega}$ of elements of $\Omega$ such that 
$$\{g \in \left[R\right] \colon \exists x \in \Omega \ g(x)=x \} = \bigcup_{i< \omega} \{g \in \left[R\right] \colon g(x_i)=x_i \} \ .$$
However, we claim that, for any countable subset $\{x_i\}_{i< \omega}$ of $\Omega$ and any $x \in \Omega \setminus \{x_i\}_{i< \omega}$, there exists $g \in \left[R\right]$ such that $g(x)=x$ and $g(x_i) \ne x_i$ for all $i$ ; granting this, we obtain the desired contraditcion. 

To conclude the proof, we briefly explain why the claim holds. Using the fact that each $x_i$ is an accumulation point of $\Gamma \cdot x_i$, one can construct inductively a sequence of clopen sets $U_j$ and elements $\gamma_j$ of $\Gamma$ with the following properties:
\begin{itemize}
\item For all $i$ $x_i \in \bigcup_{j \le i} (U_j \cup \gamma_j U_j)$.
\item For all $j$, the diameter of $U_j \cup \gamma_j U_j$ is less than $2^{-j}$, and $\gamma_j U_j \cap U_j = \emptyset$.
\item For all $j \ne k$, $(U_j \cup \gamma_j U_j) \cap (U_k \cup \gamma_k U_k) = \emptyset$.
\item For all $j$ $x \not \in \gamma_j U_j \cup U_j$ .
\end{itemize}
One can then define a bijection $g$ of $X$ by setting
$$g(y) = \begin{cases} \gamma_j(y) & \text{ if } y \in U_j \text{ for some } j \\
											\gamma_j^{-1}(y) & \text{ if } y \in \gamma_j(U_j) \text{ for some } j\\
											y & \text{ otherwise}
\end{cases} \ .$$
The fact that the diameter of $U_j \cup \gamma_j U_j$ vanishes ensures that $g$ is continuous, so $g$ belongs to $\left[R \right]$ and satisfies $g(x)=x$, $g(x_i) \ne x_i$ for all $i$.

\end{proof}

\begin{remark} 
It is not clear whether the Hausdorfness assumption is really needed: if $\left[ R \right]$ is a simple group, then a non-Hausdorff group topology is necessarily the coarse topology; indeed, the 
%Julien: Hausdorfness? Is that even a word? After googling it, it seems to be used a bit but I find it fairly ugly (I might be the one who used it originally, I don't know, but suddenly it strikes me as being really inelegant). Tomás: I supposed you had invented the word deliberately :) (to me it seems ugly-funny) Julien: !
elements that cannot be separated from $1$ by an open subset form a normal subgroup of $\left[ R \right]$. So, if $\left[R \right]$ is simple, then the above result says that the unique Baire, second countable group topology on $\left[R \right]$ is the coarse topology. However, it is an open question whether $\left[ R \right]$ is simple, even in the case when $R$ is induced by a minimal action of $\Z$.
\end{remark}

The techniques of this section are close to those employed by Rosendal in \cite{Rosendal2005}, but it seems that his results do not cover the case studied here.

%%%%%%%%%%%%%%%%%%%%%%%%%%%%%%%%%%%%%%%%%%%%%%%%%%%%%%%%%%%%%%%%%%%%%%%%%%%%%%%%%%%%%%%%%%%%%%%%%%%%%%%%%%%%%%%%%%%%%%
%%%%%%%%%%%%%%%%%%%%%%%%%%%%%%%%%%%%%%%%%%%%%%%%%%%%%%%%%%%%%%%%%%%%%%%%%%%%%%%%%%%%%%%%%%%%%%%%%%%%%%%%%%%%%%%%%%%%%%

\section{Borel Complexity of the full group of a minimal homeomorphism.}
The following question was suggested to us by T. Tsankov: what is the complexity (in the sense of descriptive set theory) of the full group of an equivalence relation induced by a minimal action of a countable group on the Cantor space? 
We answer that question for $\Gamma=\Z$. Below we use standard results and notations of descriptive set-theory, borrowed from Kechris's book \cite{Kechris1995}.

In particular, we recall that if $A$ is a countable set then a \emph{tree} on $A$ is a subset $T$ of the set $A^{<\omega}$  of finite sequences of elements of $A$, closed under taking initial segments (see \cite{Kechris1995}*{Section 2} for information on descriptive-set-theoretic trees and a detailed exposition of related notions). The set $\mathcal T$ of trees on $A$ can be endowed with a topology that turns it into a Cantor space, by setting as basic open sets all sets of the form 
$$\{T \in \mathcal T \colon \forall s \in S \ s \in T \text{ and } \forall s' \in S' \ s' \not \in T\} \ , $$ 
where $S$ and $S'$ are finite subsets of $A^{< \omega}$.

A tree $T$ is said to be \emph{well-founded} if it has no infinite branches; in this case one can define inductively the \emph{rank} of an element $s$ of $A^{< \omega}$ by setting 
$$\rho_T(s)=\sup\{ \rho(s \smallfrown a)+1 \colon s \smallfrown a \in T\} \ .$$ 
In particular, elements not in $T$ and terminal nodes in $T$ all have rank $0$; then one defines  the rank $\rho(T)$ of $T$ as being equal to the supremum of all $\rho_T(s)+1$ for $s \in A^{<\omega}$; when $T$ is nonempty, this supremum is equal to $\rho_T(\emptyset)+1$ . 

All this being said, we can begin to work, which we do by pointing out the obvious: whenever $\Gamma$ is a countable group acting by homeomomorphisms on a Cantor space $X$, the full group of the associated equivalence relation $R$ is a co-analytic subset of the Polish group $\Homeo(X)$. This is simply due to the fact that each set $\{(g,x) \in \Homeo(X) \times X \colon g(x)= \gamma \cdot x\}$ is closed, and for all $g \in \Homeo(X)$ one has
$$g \in \left[R\right] \Leftrightarrow \forall x \in X \exists \gamma \in \Gamma \ g(x)= \gamma \cdot x \ . $$
The above line shows that $\left[R \right]$ is the co-projection of an $F_{\sigma}$ subset of $\Homeo(X) \times X$, hence is co-analytic. 

One might expect that the descriptive complexity of $\left[R \right]$ is not that high in the Borel hierarchy. For instance, in the measure-preserving context, full groups are always Borel of very low complexity: it is shown in \cite{Wei2005} that the full group of an aperiodic, probability-measure-preserving equivalence relation is a $\mathbf \Pi_{3}^0$-complete subset of the group of measure-preserving automorphisms (that is, in that case the full group is a countable intersection of countable unions of closed sets). Perhaps surprisingly, it turns out that full groups of minimal homeomorphisms are not Borel. 

Below, we denote by $\varphi$ a minimal homeomorphism of a Cantor space $X$ and recall that $\left[ \varphi\right]$ denotes the full group of the associated equivalence relation. 
We also denote by $\mathcal T$ the space of all trees on $\Clop(X)$, endowed with the topology discussed above. 

\begin{defn}
To each $g \in \Homeo(X)$ we associate a tree $T_g$ on $\Clop(X)$ as follows: for any sequence $(U_0,\ldots,U_n)$ of clopen sets,
$(U_0,\ldots,U_n) $ belongs to $T_g$ iff each $U_j$ is nonempty, $U_{j+1} \subseteq U_{j}$ for all $j \in \{0,\ldots,n-1 \}$ and $g(x) \neq \varphi^{\pm j}(x)$ for all $x \in U_j$.
\end{defn}

\begin{lemma}
The map $g \mapsto T_g$ is a Borel mapping from $\Homeo(X)$ to $\mathcal T$. For any $g \in \Homeo(X)$, $g$ belongs to $\left[\varphi\right]$ if, and only if, $T_g$ is well-founded. 
\end{lemma}

\begin{proof}
We need to prove that for any finite sequence of nonempty clopen subsets $(U_0,\ldots,U_n)$ the set $\{g \in \Homeo(X) \colon (U_0,\ldots,U_n) \in T_g\}$ is Borel. For this, it is enough to show that for any nonempty clopen subset $U$ of $X$, the set $\{g \colon \forall x \in U \ g(x) \ne x \}$ is Borel. The complement of this set is $\{g \colon \exists x \in U \ g(x)=x \}$, which is closed because $U$ is clopen in $X$ and thus compact: if $g_n$ is a sequence of homeomorphisms of $X$ such that for all $n$ there exists $x_n \in U$ such that $g_n(x_n)=x_n$, and $g_n$ converges to $g$ in $\Homeo(X)$, then up to some extraction we can assume that $x_n$ converges to $x \in U$; the distance from $g(x_n)$ to $g_n(x_n)$ must converge to $0$, so $g_n(x_n)$ converges to $g(x)$, showing that $g(x)=x$. This concludes the proof that $g \mapsto T_g$ is Borel.

Next we fix $g \in \Homeo(X)$. We first assume that $g$ does not belong to $\left[ \varphi \right ]$, i.e.~there exists $x \in X$ such that $g(x) \ne \varphi^n(x)$ for all $n \in \Z$. Then, using the continuity of $g$ and $\varphi$, one can build by induction a decreasing sequence of clopen neighborhoods $U_i$ of $x$ such that for all $i$ and all $y \in U_i$ one has $g(y) \ne \varphi^{\pm i}(y)$, which yields an infinite branch of $T_g$.
Conversely, assume that $T_g$ is not well-founded and let $(U_i)_{i< \omega}$ be an infinite branch of $T_g$. Then $F= \bigcap_{i < \omega} U_i$ is nonempty, and for all $x \in F$ $g(x)$ is different from $\varphi^n(x)$ for all $n \in \Z$, showing that $g$ does not belong to $\left[ \varphi \right ]$.
\end{proof}

If $\left[\varphi\right]$ were Borel, the set $\mathcal T_{\varphi} = \{ T_g \colon g \in \left[ \varphi \right ] \}$ would be an analytic subset of $\mathcal T$, hence the boundedness principle for coanalytic ranks (see \cite{Kechris1995}*{Theorem 35.23}) would imply the existence of a countable ordinal $\alpha$ such that the rank of any element of $\mathcal T_{\varphi}$ is less than $\alpha$. We want to prove that it is not the case, so we need to produce elements of 
$\left[ \varphi \right ]$ such that the associated tree has arbitrarily large rank. Let us introduce some notation in order to make the work ahead simpler.

\begin{defn}
For any $g \in \left[ \varphi \right ]$ we let $\rho(g)$ denote the rank of $T_g$. For any finite sequence of clopen sets $(U_0,\ldots,U_n)$, we let $\rho_g(U_0,\ldots,U_n)$ denote the rank of $(U_0,\ldots,U_n)$ with regard to the tree $T_g$. If $\alpha$ and $\beta$ are ordinals, we write $\alpha\sim\beta$ to express that there are only finitely many ordinals between them; we write $\alpha\gtrsim\beta$ when $\alpha\geq\beta$ or $\alpha\sim\beta$.
\end{defn}

Note, as an encouraging sign that our introduction of $\rho$ is a good way to capture information about elements of $\left[ \varphi \right ]$, that the topological full group $\left[[ \varphi ] \right ]$ is exactly made up of all $g \in \Homeo(X)$ such that $\rho(g) < \omega$. 

In order to build elements of $\left[ \varphi \right ]$ such that the associated tree has arbitrarily large rank, the following observation will be crucial.

\begin{lemma}\label{rank_invariant_topfullgroup}
Let $g$ be an element of $\left[ \varphi \right ]$, and assume that $\rho(g) \ge \omega$. Then for any $h \in \left[[\varphi \right ]]$ one has $\rho(g)\sim\rho(hg)$.
\end{lemma}
%Since $\rho(g)=\rho(g^{-1})$, the above lemma also implies that the rank of an element $g$ not belonging to the topological full group does not change after multiplication on the right by an element of the topological full group.
Since $\rho(g)=\rho(g^{-1})$, the above lemma also holds true when multiplying on the right by an element of the topological full group.

\begin{proof}
For $k<\omega$, let $$T^k_g=\{(U_k,\dots,U_n):(U_0,\dots,U_n)\in T_g\text{ for some }U_0,\dots,U_{k-1}\}.$$ If $h\in\left[[\varphi \right ]]$, then for some $k<\omega$ and for all $x\in X$ there is $j$ such that $|j|\leq k$ and $h(x)=\varphi^j(x)$. So if $g(x)\ne\varphi^j(x)$ for all $|j|\leq n$ but $hg(x)=\varphi^m(x)$ for some $m$, then $|m|>n-k$. This implies that $T^k_g\subseteq T_{hg}$. Since $\rho(T^k_g)\geq\rho(T_g)-k$, we get $\rho(T_{hg})\geq\rho(T_g)-k$, and similarly $\rho(T_g)=\rho(T_{h^{-1}hg})\geq\rho(T_{hg})-k$, proving the claim.
%Fix $g \in \left[ \varphi \right ]$ and $h \in \left[[\varphi \right ]]$. 
%Note that $(U_0,\ldots,U_n)$ is such that $\rho_g(U_0,\ldots,U_n) \ge \omega$ if, and only if, the sequence $(U_0,\ldots,U_n)$ is decreasing and for all $i \in \N$ there exists $x \in U_n$ such that $g(x) \ne \varphi^{ j}(x)$ for all $j \in \{-i,\ldots,i\}$; since $h$ belongs to $\left[[\varphi \right ]]$ we thus see that for any $(U_0,\ldots,U_n)$
%$\rho_g(U_0,\ldots,U_n) \ge \omega \Leftrightarrow \rho_{hg}(U_0,\ldots,U_n) \ge \omega$. 
%It is then easy to prove by transfinite induction that, for any countable ordinal $\alpha \ge \omega$ and any $(U_0,\ldots,U_n)$, one has 
%$\rho_g(U_0,\ldots,U_n) \ge \alpha \Leftrightarrow \rho_{hg}(U_0,\ldots,U_n) \ge \alpha \ . $
%In turn this implies that, for any countable ordinal $\alpha \ge \omega$ and any $(U_0,\ldots,U_n)$, one has 
%$\rho_g(U_0,\ldots,U_n) = \alpha \Leftrightarrow \rho_{hg}(U_0,\ldots,U_n) = \alpha$.
\end{proof}

%In order to produce elements of $\left[ \varphi \right]$ with arbitrarily large rank, %Same phrase than before the lemma

When $U$ is a %clopen
% erase "clopen"? so that we need not state in the lemma that V is clopen... Julien: who is $V$ in the lemma? Is it $S$? Tomás: hahahaha. Sorry for the change on the lemma, it was so that the notation was continued in 4.6 (where there were already other V_i's). We can go back to V if you miss it :)
 subset of $X$ and $g$ belongs to $\left[ \varphi \right]$, we set
$$n(g,U)= \min(\{|k| \colon \exists x \in U \  g(x)= \varphi^k(x)\}) $$

\begin{lemma}\label{gluing_lemma}
Let $\alpha$ be an infinite ordinal belonging to $\{\rho(g) \colon g \in \left[ \varphi \right] \}$, and $N$ be an integer. For any nonempty clopen $U \subseteq X$, there exists 
$h \in \left[ \varphi \right]$ with support $S$ contained in $U$ (in particular, $h(U)=U$), such that $\rho(h) \gtrsim \alpha$ and $n(h,S) > N$.
\end{lemma}
% Julien: $\alpha$ being infinite is actually useless for this statement, though it is used in the proof, so I left it (the proof for the finite case simply uses minimality but I do not see the point of mentioning it)
\begin{proof}
Pick $g \in \left[ \varphi \right]$ such that $\rho(g)=\alpha$ is infinite and fix a nonempty clopen $U \subseteq X$ and an integer $N$. Using compactness of the space of $\varphi$-invariant probability measures and the fact that they are all atomless, one can find a nonempty clopen $\tilde U \subseteq U$ such that $(2N+2) \mu(\tilde U) < \mu(U)$ for any $\varphi$-invariant measure $\mu$. 

Since $\varphi$ is minimal, there exist $i_1,\ldots,i_n$ such that 
$X= \bigcup_{j=1}^n \varphi^{i_j}(\tilde U)$. For all $j \in \{1,\ldots,n\}$, denote $U_j= \varphi^{i_j}(\tilde U)$, and consider the tree $T_j$ defined by
$$(V_0,\ldots,V_n) \in T_j \Leftrightarrow (V_0,\ldots,V_n) \in T_g \text{ and } V_0 \subseteq U_j \ .$$
Denote by $\rho_j$ the rank function associated to the well-founded tree $T_j$, and by $\rho(T_j)$ the rank of $T_j$. 
For any finite sequence $(V_0,\ldots,V_k)$ of clopen subsets of $X$, we have 
$$ \left( \forall j \in \{1,\ldots,n\} \ \rho_j(V_0 \cap U_j,\ldots,V_k \cap U_j)=0 \right) \Rightarrow \rho_g(V_0,\ldots,V_k)=0\ .$$
%Also, saying that $(V_0,\ldots,V_n)$ belongs to $T_g$ is equivalent to saying that (at least) one $(V_0 \cap U_j,\ldots,V_k \cap U_j)$ belongs to $T_j$.
From this, we see by transfinite induction that 
$\rho(g)=\max \{\rho(T_j) \colon j \in \{1,\ldots,n\}\}$, so there exists $j$ such that $\rho(T_j)=\alpha$. Fix such a $j$; any element of $\left[\varphi \right]$ coinciding with $g$ on $U_j$ must have rank larger than $\alpha$. 

Applying the Glasner--Weiss result recalled as Theorem \ref{GlasnerWeiss}, we can find $f \in \left[[ \varphi \right ]]$ such that $f(U_j)=W \subseteq U$. We also have 
$$\mu(g(U_j)) < \mu\left(U \setminus \bigcup_{i=-N}^N  \varphi^i(W) \right) $$
for any $\varphi$-invariant $\mu$, so applying Theorem \ref{GlasnerWeiss} again we can find $k \in \left[[\varphi \right]]$ such that $k(g(U_j))$ is contained in $U$ and disjoint from $\bigcup_{i=-N}^N  \varphi^i(W)$.
Now, let $h$ be equal to $kgf^{-1}$ on $W$, to $f g^{-1}k^{-1}$ on $kgf^{-1}(W)$, and to the identity elsewhere. We set $S= W \cup h(W)$. 
Using the fact that $f,k$ belong to $\left[[ \varphi \right]]$ and Lemma \ref{rank_invariant_topfullgroup}, we see that $\rho(h) \gtrsim \rho(g)$. The construction ensures that $h(W)$ is disjoint from $\bigcup_{i=-N}^N  \varphi^i(W)$, so $n(h,W) >N$; since $h$ is an involution, $n(h,h(W))= n(h,W)$ is also strictly larger than $N$. This ensures that $n(h,S)>N$ and all the desired conditions are satisfied.

\end{proof}
 
 \begin{theorem}
The full group of a minimal homeomorphism of a Cantor space $X$ is a coanalytic non-Borel subset of $\Homeo(X)$.
 \end{theorem}

 \begin{proof}
Let $\varphi$ be a minimal homeomorphism of a Cantor space $X$. We explain how to produce elements of $\left[\varphi\right]$ with arbitrarily large rank.
To that end, we fix for the remainder of the proof a compatible distance on $X$, an element $g$ of $\left[\varphi\right]$ and a countable family $(V_i)$ of nonempty disjoint clopen subets of $X$
with the following property: the tree generated by terminal nodes $(U_0,\ldots,U_n)$  of $T_g$ such that $U_n=V_i$ for some $i$ has rank at least $\omega$. Note that the value $n$ associated to such a terminal node is determined by $i$: we must have $g=\varphi^{\pm (n+1)}$ on $V_i$. We note $n=N_i$, and our hyphotesis is that $(N_i)$ is unbounded.

We then pick an infinite sequence $(W_i)_{i< \omega}$ of nonempty clopen subsets of $X$ such that the diameter of each $W_i$ is less than $2^{-i}$ and $W_i \subseteq V_i$ for all $i$.
Now, let $g_i$ be any sequence of elements of $\left[\varphi\right]$ of infinite rank; using Lemma \ref{gluing_lemma}, we can find elements $h_i$ of $\left[\varphi\right]$ with support $S_i$ contained in $W_i$ and such that $\rho(h_i) \gtrsim \rho(g_i)$. We shall also ask that $n(h_i,S_i)>2N_i+1$. We then define $h \colon X \to X$ by setting 
$$h(x)=  \begin{cases} g h_i(x) &\text{ if } x \text{ belongs to some } W_i \\
g(x) & \text{ otherwise} \end{cases}\ . $$
Note that, since the sets $W_i$ % took the liberty to erase the "--'s" notation :P
 are pairwise disjoint, $h$ is well-defined. We next show that $h$ is continuous. Let $(x_i)$ be a sequence of elements of $X$ converging to some $x \in X$. If $x$ belongs to $W_j$ for some $j$ then $x_i \in W_j$ for $i$ large enough and continuity of $g$ and $h_j$ ensure that $h(x_i)$ converges to $h(x)$. So we may assume that $x$ does not belong to $\cup W_j$; in that case $h(x)=g(x)$ and since $g$ is continuous we may also assume that $x_i$ belongs to some $ W_{j_i}$ for all $i$. Each $W_i$ is clopen, so we must have $j_i \to +\infty$, hence the diameter of $W_{j_i}$ converges to $0$. Since $g$ is uniformly continuous, the diameter of $g(W_{j_i})$ also converges to $0$; $h(x_i)$ and $g(x_i)$ both belong to this set, showing that $d(g(x_i),h(x_i))$ converges to $0$. Hence $h(x_i)$ converges to $g(x)=h(x)$, proving that $h$ is continuous. The construction also ensures that $h$ is bijective, so $h$ is a homeomorphism of $X$, and $h$ belongs to $\left[\varphi\right]$. The definition of $h$ and the argument of Lemma \ref{rank_invariant_topfullgroup} (using the fact that $g= \varphi^{\pm (N_i+1)}$ on $W_i$)
%Might the reader got confused here with the extra hypothesis we have around, such as n(h_i,S_i)>2N_i+1? Indeed g is not necessary at all for this part. Julien: not sure, I thought about this proof for too long to see it from the reader's point of view. Please do any modifications you think would make it easier to understand.
 ensure that $\rho(h) \gtrsim \rho(h_i)$ for all $i$. If we have $\rho(g_i)\geq\alpha_i$ for limit ordinals $\alpha_i$, we then obtain $\rho(h) \ge \sup \alpha_i$.

%Given a sequence $g_i$ of elements of $\left[ \varphi\right ]$, we thus know how to produce an element of $\left[ \varphi\right ]$ with rank larger than $\sup \rho(g_i)$.
Now let $\alpha$ be a countable limit ordinal, and $f$ an element of $\left[ \varphi\right ]$ such that $\rho(f) \ge \alpha$; we now explain how to produce an element of $\left[ \varphi\right ]$ with rank greater than $\alpha+ \omega$, which will conclude the proof. The element in question, again denoted by $h$, is obtained by applying the construction above with $g_i=f$ for all $i$. %We keep the notations of the construction above; fix an integer $i$, consider an element of $T_h$ of the form $(U_0,\ldots,U_{n-1},V_i)$, a terminal node  $(U_0,\ldots,U_{n-1},V_i,\ldots,W)$ of $T_h$. We have $h(x)= \varphi^{\pm m}(x)$ for some $m$ and all $x \in W$, so any element of $T_{h_i}$ whose last term is $W$ must be of finite rank in $T_{h_i}$. Similarly, if a terminal node of $T_{h_i}$ ends at some $W \subseteq W_i$, then its rank inside $T_h$ must be finite. Thus $\rho_h(U_0,\ldots,U_{n-1},V_i)+\omega \ge \rho(h_i) +\omega \ge \alpha + \omega$, hence  $\rho_h(U_0,\ldots,U_{n-1},V_i) \ge \alpha$. The choice of our $V_i$'s then implies that $\rho(h) \ge \alpha + \omega$.
Let $(U_0,\dots,U_n)$, $U_0=U_n=V_i$, be a terminal node of $T_g$, so in particular $g=\varphi^{\pm (n+1)}$ on $U_0$, $n=N_i$. Since $n(h_i,S_i)>2n+1$, the only way to have $gh_i(x)=\varphi^m(x)$ for $x\in U_0$ is with $|m|>n$, and this says that $(U_0,\dots,U_n)$ belongs to $T_{gh_i}$ (and to $T_h$) as well. Once we know this, the argument of Lemma \ref{rank_invariant_topfullgroup} implies that $\rho_h(U_0,\dots,U_n)\gtrsim\rho(h_i)\geq\alpha$: if $(U'_0,\dots,U'_k)\in T_{h_i}$, $U'_0\subseteq U_n$, $2n+1<k$, then $(U_0,\dots,U_n,U'_{2n+2},\dots,U'_k)\in T_h$. We conclude that $\rho(h)\geq\alpha+n=\alpha+N_i$. Since this is true for every $i$, we get $\rho(h)\geq\alpha+\omega$, as expected.
\end{proof}

\begin{remark}
Given the result we just proved, it seems likely that the full group of an equivalence relation induced by a minimal action of a countable group $\Gamma$ on a Cantor space is never Borel. The above argument may be adapted in large part, but it is not clear to the authors how one can modify Lemma \ref{gluing_lemma} in a context where Theorem \ref{GlasnerWeiss} does not hold. 

One can nevertheless note that the above result extends to relations induced by actions of $\Z^d$ for all integers $d$, though this extension of the result is not really meaningful, indeed it is trivial once one knows that full groups associated to minimal $\Z^d$-actions are the same as full groups associated to minimal $\Z$-actions, a powerful result proved in \cite{Giordano2010}.
\end{remark}

%%%%%%%%%%%%%%%%%%%%%%%%%%%%%%%%%%%%%%%%%%%%%%%%%%%%%%%%%%%%%%%%%%%%%%%%%%%%%%%%%%%%%%%%%%%%%%%%%%%%%%%%%%%%%%%%%%%%%%
%%%%%%%%%%%%%%%%%%%%%%%%%%%%%%%%%%%%%%%%%%%%%%%%%%%%%%%%%%%%%%%%%%%%%%%%%%%%%%%%%%%%%%%%%%%%%%%%%%%%%%%%%%%%%%%%%%%%%%
\section{Closures of full groups.}
We saw above that there does not exist a Hausdorff, Baire group topology on the full group of a minimal homeomorphism $\varphi$ of a Cantor space $X$. This precludes the usage of Baire category methods; however, the same cannot be said of the closure of $\left[\varphi \right]$, which is of course a Polish group since it is a closed subgroup of $\Homeo(X)$ (and the arguments of Section \ref{topology} show that the topology induced by that of $\Homeo(X)$ is the unique Polish topology on the closure of $\left[\varphi \right]$ which is compatible with the group operations). As pointed out in \cite{Giordano1999}, the closure of $\left[\varphi \right]$ is easy to describe thanks to Theorem \ref{GlasnerWeiss}: letting $\mathcal M_{\varphi}$ denote the (compact, convex) set of all $\varphi$-invariant probability measures, we have
$$\overline{\left[\varphi \right]} =\{g \in \Homeo(X) \colon \forall \mu \in \mathcal M_{\varphi} \ g_* \mu=\mu \} \ . $$

\emph{Notation.} Below we denote the closure of the full group of $\varphi$ in $\Homeo(X)$ by~$G_{\varphi}$.
\medskip

This group is relevant when studying topological orbit equivalence of minimal homeomorphisms, because of a theorem of Giordano--Putnam--Skau which implies the following result.

\begin{prop} \label{closurefullgroup}
Let $\varphi_1$, $\varphi_2$ be two minimal homeomorphisms of a Cantor space $X$; assume that $G_{\varphi_1}$ and $G_{\varphi_2}$ are isomorphic (as abstract groups). Then $\varphi_1$ and $\varphi_2$ are orbit equivalent.
\end{prop}

 \begin{proof}
Assume that $\Phi \colon G_{\varphi_1} \to G_{\varphi_2}$ is a group isomorphism. First, the usual reconstruction techniques (see e.g.~\cite{Medynets2011}) show that there exists a homeomorphism $h \in \Homeo(X)$ such that $\Phi(g)=h g h^{-1}$ for all  $g \in G_{\varphi_1}$.

So we have that
$$\forall g \in \Homeo(X) \quad g \in G_{\varphi_1} \Leftrightarrow hgh^{-1} \in G_{\varphi_2} \ . $$
Since $\mathcal M_{\varphi_i}$ is equal to the set of measures which are invariant under translation by elements of $G_{\varphi_i}$ (for $i=1,2$), this means that 
$$\forall \mu \quad \mu \in \mathcal M_{\varphi_1} \Leftrightarrow h_* \mu \in \mathcal M_{\varphi_2}\ . $$

Then (\cite{Giordano1995}*{Theorem 2.2(iii)}) implies that $\varphi_1$ and $\varphi_2$ are orbit equivalent.\end{proof}

Of course, the converse of the above statement is true: if $\varphi_1$ and $\varphi_2$ are orbit equivalent, then their full groups are conjugated inside $\Homeo(X)$, so the closures of the full groups are also conjugated. However, the statement above is only valid a priori for actions of $\Z$: while it is true that for any minimal actions of countable groups an isomorphism between the closures of the respective full groups must be implemented by an homeomorphism of $X$, there is no reason why this homeomorphism would be sufficient to prove that the full groups themselves are isomorphic. Indeed, using ideas from ergodic theory, one can  see that there are plenty of examples of actions of countable groups $\Gamma_1$, $\Gamma_2$ on a Cantor space $X$ such that $\overline{\left[R_{\Gamma_1}\right]}=\overline{\left[R_{\Gamma_2}\right]}$, yet the two associated relations are not orbit equivalent. The example below was explained to us by D. Gaboriau. 

\begin{prop}
There exists an action of $\Z$ and an action of the free group $F_3$ on three generators on a Cantor space $X$, such that the closures of the full groups of the two actions coincide, yet the relations are not orbit equivalent.
\end{prop}

\begin{proof}
Let $\Z$ act on the Cantor space $\{0,1\}^{\omega}$ via the usual odometer map. Consider the free group $F_2$ on two generators acting by the Bernoulli shift on $\{0,1\}^{F_2}$; using a bijection between $\omega$ and $F_2$, one can see this as an action of $F_2$ on $\{0,1\}^{\omega}=X$. Let $F_3=F_2 \ast \Z$ act on $X$, where the action of $F_2$ is the Bernoulli shift and the action of $\Z$ is via the odometer map. Then the actions of $\Z$ and $F_3$ on $\{0,1\}^{\omega}$ both preserve the $(1/2,1/2)$-Bernoulli measure $\mu$ on $2^{\omega}$; since the odometer is uniquely ergodic, we see that for both actions the closure of the full group is equal to the set of all homeomorphisms which preserve $\mu$. Yet, there cannot even exist a $\mu$-preserving bijection $h$ of $X$  such that, for $\mu$-almost all $x,x' \in X$, one has 
$$(x R_\Z x') \Leftrightarrow (h(x) R_{F_3} h(x'))\ . $$
Indeed, the relation induced by the action of $\Z$ is hyperfinite, while the relation induced by the action of $F_3$ contains a subrelation which is induced by a free action of $F_2$, so it cannot be hyperfinite (see for instance \cite{Kechris2010} for information on probability-measure-preserving group actions and the properties we use here without details). Since a homeomorphism realizing an orbit equivalence between $R_\Z$ and $R_{F_3}$ would have to preserve $\mu$, we see that while the closures of both full groups coincide, the associated relations cannot be orbit equivalent.
\end{proof}

In view of this, the following question might be interesting.
\begin{question} Let $\Gamma_1$, $\Gamma_2$ be two countable \emph{amenable} groups acting minimally on a Cantor space $X$. Assume that the closures of the corresponding full groups are isomorphic as abstract groups. Must the two actions be orbit equivalent?
\end{question}

If one knew that any minimal action of a countable amenable group is orbit equivalent to a $\Z$-action then the answer to the question above would be positive; the result of \cite{Giordano2010} mentioned at the end of the previous section implies that the above question has a positive answer when $\Gamma_1,\Gamma_2$ are finitely generated free abelian groups.

We already pointed out that it is unknown whether the full group of a minimal homeomorphism $\varphi$ is simple. This reduces to deciding whether the full group coincides with its derived subgroup. Indeed, it is proved in \cite{Bezuglyi-Medynets2008}*{Theorem 3.4} that any normal subgroup of $\left[\varphi\right]$ contains its derived subgroup; the same is true for $G_\varphi$, as can be seen by following the proof of \cite{Bezuglyi-Medynets2008}.

Unfortunately, it seems to be hard in general to decide which elements of $\left[\varphi\right]$ are products of commutators (though one might conjecture that every element has this property; partial results in this direction can be found in \cite{Bezuglyi-Medynets2008}). The use of Baire category methods might make things simpler in the case of $G_{\varphi}$, especially in view of the following folklore result.

\begin{prop} \label{comeager_commutator}
Let $G$ be a Polish group; assume that $G$ has a comeager conjugacy class. Then every element of $G$ is a commutator.
\end{prop}

\begin{proof}
Assume that $\Omega$ is a comeager conjugacy class in $G$, and let $g \in G$. The intersection $g \Omega \cap \Omega$ is nonempty; picking an element $g_0$ in this intersection, we see that there exists $k \in G$ such that $g k g_0 k^{-1}=g_0$, in other words $g= g_0 k g_0^{-1}k^{-1}$.
\end{proof}

It is thus interesting to understand when $G_{\varphi}$ has a comeager conjugacy class, even more so because of the following observation.

\begin{prop}
Let $\varphi$ be a minimal homeomorphism of a Cantor space $X$, and assume that $G_{\varphi}$ has a comeager conjugacy class. Then $G_{\varphi}$ has the \emph{automatic continuity property}, i.e.~any homomorphism from $G_{\varphi}$ to a separable topological group is continuous. 
\end{prop}

 \begin{proof}
The argument in \cite{Rosendal2007}*{Theorem 12} adapts straightforwardly.
 \end{proof}

In the next section, we will discuss in more detail the problem of existence of comeager conjugacy classes in $G_\varphi$ in the particular case when $\varphi$ is uniquely ergodic, recovering in particular a result of Akin that provides many examples of this phenomenon.

In the first version of this article, we proved a weaker version of the result below, which worked only in the case when $n=1$ and $\varphi$ is a uniquely ergodic homeomorphism (see the next section); we are grateful to K. Medynets for pointing out to us the following stronger result.

\begin{theorem}[Grigorchuk--Medynets \cite{Grigorchuk2012}] \label{HrushovskiStrong}
Let $\varphi$ be a minimal homeomorphism of a Cantor space $X$. Then $\{(g_1,\ldots,g_n) \colon (g_1,\ldots,g_n) \text{ generates a finite group} \}$ is dense in $\left[\varphi\right]^n$ for all $n$. 
\end{theorem}

\begin{proof}
Since this statement is not explicitly written down in \cite{Grigorchuk2012} (though it is very close to Theorem 4.7 there), we describe the argument for the reader's convenience.
We simply prove that the set of elements of finite order is dense in $[\varphi]$; the proof of the general case is an easy consequence of this argument.

We may assume, replacing $\varphi$ by a minimal homeomorphism which is orbit equivalent to it (which does not affect the full group) that the topological full group $\left[[\varphi] \right]$ is dense in $\left[\varphi \right]$. This fact is pointed out in \cite{Bezuglyi2002}*{Theorem 1.6}, and follows from a combination of \cite{Glasner1995a}*{Theorem 2.2} and \cite{Giordano1999}*{Lemma 3.3}. Under this assumption, we only need to prove that the set of elements of finite order is dense in  $\left[[\varphi] \right]$. 

We fix $\gamma \in \left[[\varphi] \right]$. We let $D_i=\{x \colon \gamma(x)=\varphi^{i}(x)\}$, and 
$E_{k}=\{x \colon \gamma^{-1}(x)=\varphi^{k}(x)\}$ . The sets $D_i$ form a clopen partition of $X$, as do the sets $E_k$ ; we let $j_\gamma \colon X \to \Z$ (resp. $k_\gamma$) be the continuous function defined by $j_\gamma(x)=j$ iff $x \in D_j$ (resp. $k_\gamma(x)=k$ iff $x \in E_k$). We also pick $K$ such that $D_j=\emptyset=E_j$ for all $|j| > K$, fix a compatible distance $d$ on $X$ for the remainder on the proof, and let $\delta>0$ be such that $d(D_i,D_j) > \delta$ for all nonempty $D_i \ne D_j$, and $d(E_i,E_j) > \delta$ for all nonempty $E_i \ne E_j$.

Recall that a \emph{Kakutani--Rokhlin partition} associated to $\varphi$ is a clopen partition of $X$ of the form $\{\varphi^i (B_n) \colon 0 \le n \le N, 0 \le i \le h_{n}-1\}$. 
The \emph{base} of the partition is $B= \bigcup_{n=0}^N B_n$, while its \emph{top} is $T=\bigcup_{n=0}^N \varphi^{h_n-1} B_n$. Note that $\varphi(T)=B$. For all $i$, we set 
\begin{gather*}
Y_i= \bigcup_{n=0}^{N} \varphi^i (B_n)\quad \text{and}\quad  Z_i= \bigcup_{n=0}^{N} \varphi^{-i} (B_n)
\end{gather*}

K-R partitions exist because $T$ is minimal; actually, one can use minimality to ensure that the following conditions are satisfied (see for instance \cite{Grigorchuk2012} for a discussion on these partitions and references):
\begin{enumerate}
\item The functions $j_{\gamma}$ and $k_\gamma$ are constant on each atom of the partition.
\item $\min \{h_n \colon 0 \le n \le N\} \ge 2K+2$.
\item The diameter of each $Y_i$ and each $Z_i$ is less than $\delta$ for all $i \in \{0,\ldots K\}$ (this can be ensured because of the uniform continuity of $\varphi$ and $\varphi^{-1}$, and the fact that one can build K-R partitions whose base has arbitrarily small diameter).
\end{enumerate}
Fix such a partition. Note that the third condition ensures that $j_{\gamma}$ and $k_\gamma$ are constant on each $Y_i, Z_i$ ($|i| \le K$), and the second condition guarantees that the sets 
$(Y_i)_{0 \le i \le K}$, $(Z_i)_{1 \le i \le K}$ are pairwise disjoint. 
We now define $P \in \left[[ \varphi \right]]$ as follows. For all $n$, and all $i \in \{0,\ldots,h_{n-1}\}$, let $j_{\gamma}(n,i)$ be the value of $j_{\gamma}$ on $\varphi^{i}(B_n)=B_{n,i}$.
\begin{itemize}
\item If $0 \le j_{\gamma}(n,i)+i \le h_{n-1}$, then $P(x)= \gamma(x)$ for all $x \in B_{n,i}$.
\item If $j_{\gamma}(n,i)+i < 0$, then necessarily $i<K$, so $B_{n,i}=Y_i$, and since $j_{\gamma}$ is constant on $Y_i$ one has $\gamma(Y_i) \subseteq Z_l$ for some $1\leq l\leq K$. The inclusion must be an equality since $k_{\gamma}$ is constant on $Z_l$. Then set
$P(x)=\varphi^{-i+h_n-l }(x)$ for all $x \in B_{n,i}$.
% Vérifier les constantes!
\item If $j_{\gamma}(n,i)+i \ge h_n$, then one must similarly have $B_{n,i}=Z_j$ for some $1\leq j\leq K$, $\gamma(Z_j)=Y_l$ for some $0\leq l<K$, and one can set 
$P(x)= \varphi^{-i+l}(x)$ for all $x \in B_{n,i}$.
\end{itemize}

It is straightforward to check that $P$ has finite order; also, the fact that the diameter of each $Y_i$, $Z_i$ for $|i| \le K$ is small ensures that for all $x$ one has both $d(P(x),\gamma(x)) \le \delta$ and $d(P^{-1}(x),\gamma(x)) \le \delta$.
Thus $\gamma$ belongs to the closure of the set of elements of finite order, which concludes the proof.
\end{proof}

Actually, as pointed out by K. Medynets, this argument shows that $[\varphi]$ contains a dense locally finite subgroup (the group of all elements which preserve a positive semi-orbit %this same group even if [[\varphi]] is not dense in [\varphi]?
; see the remarks in \cite{Grigorchuk2012}*{Section 5}). We will not need this fact so do not give any details.

Even though we do not know whether $G_{\varphi}$ or $\left[\varphi\right]$ are simple in general, we can use Theorem \ref{HrushovskiStrong} to prove that these groups do not have any non-trivial closed normal subgroups.

\begin{theorem}
Let $\varphi$ be a minimal homeomorphism of a Cantor space $X$. Then $G_{\varphi}$ and $\left[\varphi\right]$ are topologically simple.
\end{theorem}

\begin{proof}
We show that the derived subgroup of $\left[\varphi\right]$ is dense, which implies the simplicity of both groups by the result of Bezuglyi--Medynets recalled in the paragraph before Proposition \ref{comeager_commutator}. Say that $g \in \left[\varphi\right]$ is a \emph{$p$-cycle} on a clopen set $U$ if $U$ is the support of $g$, $g^p=1$ and there exists a clopen $A$ such that $U= A \sqcup g(A) \ldots \sqcup g^{p-1} A$ (this is the same as saying that the $g$-orbit of every element of $U$ has cardinality $p$, and every element outside $U$ is fixed by $g$). 
Theorem \ref{HrushovskiStrong} implies that products of cycles are dense in $\left[\varphi\right]$, so it is enough for our purposes to show that $p$-cycles are products of commutators for any integer $p$.

Let $g$ be a $p$-cycle on a clopen $U$, with $U=\sqcup_{i=0}^{p-1}g^i(A)$. Given a permutation $\sigma$ belonging to the permutation group $S_p$ on $p$ elements, we denote by $g_{\sigma}$ the element of $\left[\varphi\right]$ defined by setting $g_{\sigma}(x)=x$ for all $x$ outside $U$ and 
$$\forall i \in \{0,\ldots,p-1\} \ \forall x \in g^{i}(A) \quad g_{\sigma}(x)=g^{\sigma(i)-i}(x)\ .$$
The map $\sigma \mapsto g_{\sigma}$ is a homomorphism from $S_p$ to $\left[\varphi\right]$.
Since the commutator subgroup of $S_p$ is the alternating subgroup $A_p$, we thus see that whenever $\sigma$ belongs to $A_p$ $g_{\sigma}$ is a product of commutators. In particular, $g$ has this property if $p$ is odd. If $p$ is even, let $\tau$ be the transposition of $S_p$ which exchanges $0$ and $1$. Then \cite{Bezuglyi-Medynets2008}*{Corollary 4.8} tells us that $g_\tau$ is a product of $10$ commutators in $\left[\varphi\right]$; since $gg_\tau=g_{\sigma}$ for some $\sigma \in A_p$, $g$ is also a product of commutators.

\end{proof}
Let us mention another reason why we think it might be interesting to further study the properties of closures of full groups.

\begin{prop}
Let $\varphi$ be a minimal homeomorphism of a Cantor space $X$. Then $G_\varphi$ is an amenable Polish group.
\end{prop}

\begin{proof}
By Theorem \ref{HrushovskiStrong} there exists an increasing sequence of compact subgroups of $G_\varphi$ whose union is dense in $G_\varphi$ (see \cite{Kechris2007}*{Proposition 6.4}, which must then be amenable.

The result would also follow immediately from the stronger fact that $G_\varphi$ actually contains a dense locally finite subgroup.

\end{proof}

This fact is particularly interesting in view of a question of Angel--Kechris--Lyons \cite{Kechris2012}*{Question 15.1} asking whether, whenever an amenable Polish group has a metrizable universal minimal flow, the universal minimal flow is uniquely ergodic. A positive answer to the following problem would then show that the answer to Angel--Kechris--Lyons' question is negative.

\begin{question} Let $\varphi$ be a minimal homeomorphism of a Cantor space $X$. Is the universal minimal flow of $G_{\varphi}$ metrizable?
\end{question}

% One day we meet I will ask you to tell me more about this part.
\begin{remark}% I wanted to make remarks follow the same enumeration [Theorem] but for some reason I cannot get the .sty (nor the .bib) to compilate.
Proving that there exists \emph{one} minimal homeomorphism $\varphi$ which is not uniquely ergodic, yet has a metrizable universal minimal flow would be enough to answer negatively the question of Angel--Kechris--Lyons mentioned above.  In the opposite direction, proving that the universal minimal flows of these groups are not metrizable as soon as the homeomorphism is not uniquely ergodic, and are metrizable otherwise, would point towards a positive answer to their question. 

At the moment, this seems out of reach: for instance, when $\varphi$ is equal to the usual binary odometer, $G_{\varphi}$ is just the set of all homeomorphisms of the Cantor space $\{0,1\}^\omega$ which preserve the usual  $(1/2,1/2)$-Bernoulli measure on $\{0,1\}^\omega$. Identifying the universal minimal flow of this group is already a very complicated problem, studied in \cite{Kechris2012a} where a candidate (which is metrizable) is proposed. Thus it seems that the current state of the art does not, for the moment, allow us to hope for an easy answer to our 
question.
\end{remark}

\section{Uniquely ergodic homeomorphisms and Fra\"iss\'e theory}
From now on, we focus on the case when $\varphi$ is uniquely ergodic, i.e.~there is a unique $\varphi$-invariant probability measure. 

\begin{defn}
A Borel probability measure $\mu$ on a Cantor space $X$ is said to be a \emph{good measure} if $\mu$ is atomless, has full support, and satisfies the following property: whenever $A,B$ are clopen subsets of $X$ such that $\mu(A) \le \mu(B)$, there exists a clopen subset $C$ of $B$ such that $\mu(C)=\mu(A)$.
\end{defn}

Note that in the definition above the fact that $A,B,C$ are clopen is essential. Good measures are relevant in our context because of the following fact.

\begin{theorem}[\cite{Akin2005}; Glasner--Weiss \cite{Glasner1995a}] \label{measure and homeo}
Let $\mu$ be a probability measure on a Cantor space $X$. There exists a minimal homeomorphism $\varphi$ of $X$ such that $\{\mu\}= \mathcal M_{\varphi}$ if, and only if, $\mu$ is a good measure.
\end{theorem}
The fact that the goodness of $\mu$ is a necessary condition in the result above is due to Glasner--Weiss (it follows directly from the result we recalled as Theorem \ref{GlasnerWeiss}); the fact that is is sufficient is due to Akin.

It seems natural to ask the following question, which we only mention in passing. 
\begin{question}
Can one give a similar characterization of compact, convex subsets $K$ of the set of probability measures on a Cantor space $X$ for which there exists a minimal homeomorphism $\varphi$ of $X$ such that $K$ is the set of all $\varphi$-invariant measures? 
\end{question}

The following invariant of good measures is very useful.

\begin{defn}[Akin \cite{Akin2005}]
Let $\mu$ be a good measure on a Cantor space $X$. Its \emph{clopen value set} is the set 
$$V(\mu)= \{r \in [0,1] \colon r= \mu(A) \text{ for some clopen } A \subseteq X\} \ . $$
\end{defn}

A good measure $\mu$ on a Cantor space $X$ is completely characterized by its clopen value set, in the sense that for any two good measures $\mu, \nu$ on $X$ with the same clopen value set there must exist a homeomorphism $g$ of $X$ such that $g_* \mu=\nu$ (see \cite{Akin2005}*{Theorem~2.9}; we discuss a different proof below). If $\mu$ is a good measure, then $V(\mu)$ is the intersection of a countable subgroup of $(\R,+)$ and $[0,1]$, contains $1$, and is dense in the interval; conversely it is not hard to see that any such set is the clopen value set of some good measure $\mu$. The density condition corresponds to the fact that $\mu$ is atomless, and is equivalent (since $1 \in V$) to saying that $V$ is not contained in $\frac{1}{p} \Z$ for any integer $p$.

\begin{defn}
Given a good measure $\mu$ on a Cantor space $X$, we follow \cite{Akin2005} and denote by $H_{\mu}$ the set of all homeomorphisms of $X$ which preserve $\mu$. For a countable $V \subset [0,1]$, we denote by $\langle V \rangle$ the intersection of the subroup of $(\R,+)$ generated by $V\cup \{1\}$ with $[0,1]$; we say that $V$ is \emph{group-like} when $V$ is not contained in $\frac{1}{p} \Z$ for any integer $p$ and $V= \langle V \rangle$. In that case, we denote by $\mu_V$ the good measure whose clopen value set is equal to $V$.
\end{defn}

Of course, $\mu_V$ above is only defined up to isomorphism; since we focus on isomorphism-invariant properties we allow ourselves this small abuse of terminology.

We would like to understand when there exists a comeager conjugacy class in $H_{\mu}$. Akin \cite{Akin2005}*{Theorem 4.17} proved that this holds true whenever $V(\mu) + \Z$ is a $\Q$-vector subspace of $\R$, or equivalently whenever any clopen subset can be partitioned into $m$ clopen subsets of equal measure for any integer $m$. One can check that this also holds true, for instance, when $\mu$ is a Bernoulli measure (this is explicitly pointed out in \cite{Kechris2007}), and it was our hope that this property would be satisfied by all good measures. Unfortunately, such is not the case, as we will see shortly; since we approach this problem via techniques developed by Kechris--Rosendal \cite{Kechris2007}, we quickly recall the framework for their results.

A \emph{signature} $L$ is a set $\{\{(f_i,n_i)\}_{i \in I}, \{(R_j,m_j)\}_{j \in J}, \{c_k\}_{k \in K}\}$
 where each $f_i$ is a \emph{function symbol} of arity $n_i$, each $R_j$ is a \emph{relation symbol} of arity $m_j$, and each $c_k$ is a \emph{constant symbol}. 

Given a signature $L$, an \emph{$L$-structure} $\mathcal M$ consists of a set $M$ along with a family 
$\{\{(f_i^{\mathcal M})\}_{i \in I}, \{R_j^{\mathcal M}\}_{j \in J}, \{c_k^{\mathcal M}\}_{k \in K}\}$ where each $f_i^{\mathcal M}$ is a function from $M^{k_i}$ to $M$, each $R_j^{\mathcal M}$ is a subset of $M^{m_j}$, and each $c_k^{\mathcal M}$ is an element of $M$. In our context, one might for instance consider the signature containing constant symbols $0$ and $1$, binary functional symbols $\wedge$ and $\vee$, and consider the class of structures in that signature which are boolean algebras with minimal element (the empty set) corresponding to the constant $0$, and maximal element (the whole set) corresponding to the constant $1$. It might also simplify matters to add a unary function symbol standing for complementation. Here, we are not concerned merely with boolean algebras, but with probability algebras. One way to fit those into our framework is to first fix a set $V \subseteq [0,1]$ (the set of values allowed for the probability measure), and add a unary predicate $\mu_v$ for each $v \in V$. Then, one can naturally consider the class of probability algebras with measure taking values in $V$ as a class of structures in this signature $L_V$.

There are natural notions of embedding/isomorphism of $L$-structures. Assume that we have fixed a countable signature $L$ (that is, each set $I$, $J$, $K$ above is at most countable), and that $\mcK$ is a class of finite $L$-structures. Then one says that $\mcK$ is a \emph{Fra\"iss\'e class} if it satisfies the four following conditions:
\begin{enumerate}
\item $\mcK$ countains only countably many structures up to isomorphism, and contains structures of arbitrarily large finite cardinality.
\item $\mcK$ is \emph{hereditary}, i.e.~if $A \in \mcK$ and $B$ embeds in $\mcK$, then $B \in \mcK$.
\item $\mcK$ satisfies the \emph{joint embedding property} (JEP), that is, any two elements of $\mcK$ embed in a common element of $\mcK$.
\item $\mcK$ satisfies the \emph{amalgamation property} (AP), that is, given $A,B,C \in \mcK$ and embeddings $i \colon A \to B$, $j \colon A \to C$, there exists $D \in \mcK$ and embeddings $\beta \colon B \to D$ and $\gamma \colon C \to D$ such that $\beta \circ i= \gamma \circ j$.
\end{enumerate}

The point is that, given a Fra\"iss\'e class $\mcK$, there exists a unique (up to isomorphism) $L$-structure $\bK$ whose \emph{age} is $\mcK$ and which is \emph{homogeneous}. Here, the age of a structure is the class of finite $L$-structures which embed in it, and a structure $\bK$ is homogeneous if any isomorphism between finite substructures of $\bK$ extends to an automorphism of $\bK$. Conversely, if $\bK$ is a countable homogeneous $L$-structure whose finitely generated substructures are finite, then its age is a Fra\"iss\'e class. 

For instance, the class of finite boolean algebras is a Fra\"iss\'e class and its limit is the unique countable atomless Boolean algebra, whose Stone space is the Cantor space ---so the automorphism group of the limit is just the homeomorphism group of the Cantor space in another guise. Note that the automorphism group of any countable structure $\bK$ may be endowed with its permutation group topology, for which a basis of neighborhoods of the neutral element is given by pointwise stabilizers of finite substructures. 

Let us fix a good measure $\mu$ on a Cantor space $X$, set $V=V(\mu)$, and consider the probability algebra $(\Clop(X),\mu)$ made up of all clopen subsets of $X$ endowed with the measure $\mu$, in the signature $L_V$ discussed above. Then it follows from Theorems \ref{measure and homeo} and \ref{GlasnerWeiss} that this is a homogeneous structure: any measure-preserving isomorphism between two finite clopen subalgebras of $X$ is induced by a measure-preserving homeomorphism of $X$, i.e.~an automorphism of the boolean algebra $\Clop(X)$ which preserves the measure $\mu$. Also, an easy induction on the cardinality of finite subalgebras of $(\Clop(X),\mu)$ shows that its age consists of the finite probability algebras whose measure takes values in $V$. Hence this is a Fra\"iss\'e class; note that this implies that two good measures $\mu_1,\mu_2$ such that $V(\mu_1)=V(\mu_2)$ must be isomorphic, by the uniqueness of the Fra\"iss\'e limit (this was first proved by Akin \cite{Akin2005}).

Now we can come back to the question of existence of dense/comeager conjugacy classes in $H_\mu$, when $\mu$ is a good measure. Assume again that $\mcK$ is a Fra\"iss\'e class in some countable signature $L$, let $\bK$ be its Fra\"iss\'e limit and let $\mcK_1$ denote the class of structures of the form $(A,\varphi)$, where $A$ belongs to $\mcK$ and $\varphi$ is a \emph{partial} automorphism of $A$, i.e.~an isomorphism from a substructure of $A$ onto another substructure of $A$. An embedding between two such structures $(A,\varphi)$ and $(B,\psi)$ is an embedding $\alpha$ of $A$ into $B$ such that $\psi\circ\alpha$ extends $\alpha\circ\varphi$. Then, the existence of a dense conjugacy class in $\Aut(\bK)$ is equivalent to saying that the class $\mcK_1$ satisfies the joint embedding property (see \cite{Kechris2007}*{Theorem 2.1}).

The existence of a comeager conjugacy class is a bit harder to state. Keeping the notations above, say that a class of structures $\mcK$ satisfies the \emph{weak amalgamation property} if for any $A \in \mcK$ there exists $B \in \mcK$ and an embedding $i \colon A \to B$ such that for any $C,D \in \mcK$ and any embeddings $r \colon B \to C$, $s\colon B \to D$, there exists $E \in \mcK$ and embeddings $\gamma \colon C \to E$ and $\delta \colon D \to E$ such that $\gamma \circ r \circ i= \delta \circ s \circ i$. Then \cite{Kechris2007}*{Theorem 3.4} states that there exists a comeager conjugacy class in $\Aut(\bK)$ if and only if $\mcK_1$ satisfies both (JEP) and (WAP).

We now know what combinatorial properties to study when looking at the automorphism groups of good measures; fix a good measure $\mu$ and consider the corresponding Fra\"iss\'e class $\mcK^\mu$, which is made up of all finite probability algebras whose measure takes its values inside $V(\mu)$. Theorem \ref{HrushovskiStrong} provides a good starting point: indeed, it shows that any element of $\mcK^\mu_1$ can be embedded in an element of the form $(A,\varphi)$, where $\varphi$ is a \emph{global} automorphism of $A$. We denote this class by~$\mcK^\mu_{\text{aut}}$.

We now provide our original proof of the fact that $\mcK^\mu_{\text{aut}}$ is cofinal in $\mcK^\mu_1$; while the result below is weaker than Theorem \ref{HrushovskiStrong}, we feel that the proof might be of some independent interest, since the argument is really about good measures and could conceivably be useful in some other contexts.

\begin{theorem} \label{Hrushovski}
Let $\varphi$ be a uniquely ergodic homeomorphism of a Cantor space $X$. Then the set $\{g \in \left[\varphi \right] \colon g \emph{ has finite order} \}$ is dense in $\left[\varphi\right]$.
\end{theorem}

\begin{proof}
We denote by $\mu$ the unique $\varphi$-invariant measure. Pick $g\in \left[\varphi\right]$ and let $U_0,\dots,U_{n-1},$ $V_0,\dots,V_{n-1}$ be clopen partitions of $X$ such that $g(U_i)=V_i$, (thus $\mu(U_i)=\mu(V_i)$), for $i=0,\dots,n-1$. We have to find a finite order homeomorphism $h\in \left[\varphi\right]$ also satisfying $h(U_i)=V_i$ for each $i$. Let $B_0,\dots,B_{m-1}$ be an enumeration of the atoms of the finite boolean algebra generated by the sets $U_i,V_i$. We will find a further finer clopen partition $\mcB'$ and a measure-preserving permutation $h':\mcB'\to\mcB'$ such that $h'(B')\subseteq V_i$ for $B'\subseteq U_i$, whence the desired $h$ can be obtained by applying Theorem \ref{GlasnerWeiss} to find an element of $\left[\varphi \right]$ mapping $B'$ to $h'(B')$ for all $B' \in \mcB'$.

It will be useful to introduce the following terminology: a clopen subset $S\subseteq V_i$ is an \emph{initial fragment} of $V_i$ if there is $j_0<m$ such that $B_j\cap V_i\subseteq S$ for every $j<j_0$ and $B_j\cap S=\emptyset$ for $j_0<j$.

To construct the desired $\mcB'$ and $h'$ we start by building a family of clopen sets $T_0=\{C_\sigma\in \Clop(X):\sigma\in m^{<\omega}\}$, each of its elements contained in some atom $B_j$, with the following properties:
\begin{enumerate}
\item $C_\emptyset=\emptyset$, $C_j=V_0\cap B_j$,
\item if $C_\sigma\subseteq U_i$, $\sigma\neq\emptyset$, then $C_{\sigma j}\subseteq V_i\cap B_j$,
\item if $C_\sigma\subseteq U_0$, $\sigma\neq\emptyset$, then $C_{\sigma j}=\emptyset$ for each $j$,
\item\label{muC} if $C_\sigma\nsubseteq U_0$, then $\mu(C_\sigma)=\mu(\bigcup_{j<m} C_{\sigma j})$,
\item\label{disj} if $\sigma\neq\tau$, then $C_\sigma\cap C_\tau=\emptyset$,
\item\label{infrag} for $k<\omega$, $\bigcup\{C_\sigma:|\sigma|\leq k, C_\sigma\subseteq V_i\}$ is an initial fragment of $V_i$,
\item\label{meas} for $0<i< n$ and $k<\omega$, $$\mu(\bigcup\{C_\sigma:|\sigma|<k, C_\sigma\subseteq U_i\})=\mu(\bigcup\{C_\sigma:|\sigma|\leq k, C_\sigma\subseteq V_i\}),$$
\item\label{pres} for $0<k<\omega$, $$\mu(U_0)=\mu(\bigcup\{C_\sigma:|\sigma|=k,\text{ or else }|\sigma|<k\text{ and }C_\sigma\subseteq U_0\}),$$
\item\label{empty} $C_{\sigma}$ is empty for all but finitely many $\sigma$.
\end{enumerate}
Note that the above conditions %(1)-\eqref{pres}
 imply that the set $\mcS= \{\emptyset\} \cup \{\sigma \colon C_{\sigma} \ne \emptyset\}$ must then be a subtree of $m^{< \omega}$, which is finite from condition \eqref{empty}.

The starting point is given by the first item. Suppose we have defined $C_\sigma$ for all $\sigma$ of length $|\sigma|\leq k$, $k\geq 1$, and assume inductively items (2)-(\ref{meas}) to hold so far. Let $$E^k_i=\{C_\sigma:|\sigma|=k, C_\sigma\subseteq U_i\},\ F^k_i=\{C_\sigma:|\sigma|\leq k, C_\sigma\subseteq V_i\}.$$ 
For $C_\sigma\in E^k_0$ we define $C_{\sigma j}=\emptyset$ for each $j$. For $i>0$, items (\ref{disj}), (\ref{meas}) and the fact that $\mu(U_i)=\mu(V_i)$ and that $\mu$ is a good measure allow us to find a clopen subset $S$ of $V_i\setminus\bigcup F^k_i$ of measure $\mu(\bigcup E^k_i)$. Again because $\mu$ is a good measure, we can partition $S$ into clopen sets $S_\sigma$ such that $\mu(S_\sigma)=\mu(C_\sigma)$ for each $C_\sigma\in E^k_i$;
 we then define $C_{\sigma j}=S_\sigma\cap B_j$. One should only take care to choose $S$ so that (\ref{infrag}) be true for $k+1$; the rest of the items (2)-(\ref{pres}) are preserved automatically. 
So, as for the choice of $S$, let $j_0<m$ be least such that $$\mu(\bigcup F^k_i)+\mu(\bigcup E^k_i)\leq\mu(\bigcup_{j\leq j_0} B_j\cap V_i).$$ Then 
choose ---once more, because $\mu$ is good--- a clopen subset $S_0\subseteq B_{j_0}\setminus\bigcup F^k_i$ of measure $\mu(\bigcup_{j\leq j_0} B_j\cap V_i)-\mu(\bigcup F^k_i)-\mu(\bigcup E^k_i)$, 
and define $S=((B_{j_0}\setminus S_0)\cup\bigcup_{j<j_0}(B_j\cap V_i))\setminus\bigcup F^k_i$. We~should remark that it is here that we need to restrict to a single good measure (thus to a uniquely ergodic homeomorphism); we do not know whether the argument can be adapted to a \emph{good family} of measures.

The tree already constructed, we should argue that item (\ref{empty}) holds. Let $b_k$ be the number of atoms $B_j$ intersecting $\bigcup\{C_\sigma:|\sigma|\leq k\}$; let $c_k$ be the number of non-empty sets $C_\sigma$ with $|\sigma|=k$, plus the number of non-empty sets $C_\sigma\subseteq U_0$ with $|\sigma|<k$. By construction, especially item (\ref{infrag}), $c_k<c_{k+1}$ only if $b_k<b_{k+1}$: a set $C_\sigma$ has more than one non-empty successor only if some new atom $B_j$ is attained. As $\{b_k\}$ is bounded by $m$, $\{c_k\}$ is bounded. This means that, after a certain point, each node has at most one non-empty successor, which must be of equal measure. Since distinct nodes are disjoint, this is possible only if (\ref{empty}) holds.

The family $\{C_\sigma \colon \sigma \text{ is a leaf of } \mcS \}$ forms a finite partition of $U_0$. Indeed: they are disjoint, by (\ref{disj}). If a non-empty node $N$ is not inside $U_0$, then it has a non-empty successor by (\ref{muC}); thus the leaves lie inside $U_0$. Finally, they add up the measure of $U_0$, by (\ref{pres}), and a clopen set of $\mu$-measure zero must be empty.

Starting from the leaves and following the tree backwards, we now introduce clopen partitions of the elements of $T_0$, and thus a partition of $\bigcup T_0$ itself; this will be an $h'$-invariant part of our intended refinement $\mcB'$, so we simultaneously define how $h'$ acts on this part. Let $\rho:m^{<\omega}\to\omega$ be the rank function of $\mcS$. Zero-rank nodes (the leaves) get the trivial partition; we postpone the definition of $h'$ on them, which will be given at the end of the construction. If we have already defined a partition $\mcP(\tau)$ of $C_\tau$ for each $\tau$ such that $\rho(\tau)<r$, and $\rho(\sigma)=r$, we use that $\mu$ is a good measure to get a partition $$\mcP(\sigma)=\{P_Q\subseteq C_\sigma:Q\in\mcP(\sigma j), j=0,\dots,m-1\}$$ such that $\mu(P_Q)=\mu(Q)$ for each atom $P_Q$. At the same time we define $h'(P_Q)=Q$. Moreover, we can associate an element $L(P)$ to each $P\in\mcP(\sigma)$, so that $L(C_\sigma)=C_{\sigma}$ if $\sigma$ is a leaf of $\mcS$, and $L(P_Q)=L(Q)$ in the inductive step. We stop the induction just before the root $\emptyset$, and finally define, for each leaf $\sigma$, $h'(C_\sigma)=P$ if $P$ is the (only) atom in the families $\mcP(0),\dots,\mcP(m-1)$ such that $L(P)=C_\sigma$.

Collecting the atoms in each $\mcP(\sigma)$ we now have a partition of $\bigcup T_0$, which we take to be the restriction of $\mcB'$ to this set; we know how to define $h'$ on this part. The whole process must now be repeated on $X\setminus\bigcup T_0$, starting with the sets $U'_i=U_i\setminus\bigcup T_0$, $V'_i=V_i\setminus\bigcup T_0$ and $B'_j=B_j\setminus\bigcup T_0$, the set $U'_1$ playing the role of $U_0$. Remark simply that $\mu(U'_i)=\mu(V'_i)$, by (\ref{meas}). And so on; after less than $n$ iterations of this process we reach the definition of the intended $\mcB'$ and $h'$.
\end{proof}

\begin{remark} The above result might be rephrased as follows: for a good measure $\mu$, a \emph{partial} $\mu$-automorphism of a finite algebra $\mcA\subset\Clop(X)$ can always be extended to a $\mu$-automorphism of a finite algebra $\mcB\subset\Clop(X)$ extending $\mcA$.\end{remark}

Now, in order to understand when $\mcK^\mu_1$ satisfies (JEP), we only need to think of automorphisms of finite algebras; explicitly, we now see that $H_\mu$ has a dense conjugacy class if and only if the following condition is satisfied: whenever $\mathcal A, \mathcal B$ are finite subalgebras of $\Clop(X)$, and $a,b$ are automorphisms of $(\mathcal A,\mu)$ ,$(\mathcal B,\mu)$ respectively, there exists a finite subalgebra $\mathcal C$ of $\Clop(X)$ and an automorphism $c$ of $(\mathcal C,\mu)$ such that there exist $\mu$-preserving embeddings $\alpha \colon \mathcal A \to \mathcal C$ and $\beta \colon \mathcal B \to \mathcal C$ satisfying $c(\alpha(A))= \alpha(a(A))$ for all $A \in \mathcal A$, and $c(\beta(B))=\beta(b(B))$ for all $B \in \mathcal B$. 

Unfortunately, this property is not always satisfied. Indeed, assume that $\mu$ satisfies (JEP), and that there exists $A \in \Clop(X)$ such that $ \mu(A)=\frac{1}{n}$ for some integer $n$. Then, there exists an element $a \in H_{\mu}$ such that $X$ is the disjoint union of $A,\ldots,a^{n-1}(A)$. 
Let now $r$ be any element of $V(\mu)$, $B$ a clopen subset of $X$ such that $\mu(B)=r$,  and consider: 
\begin{itemize}
\item the algebra $\mathcal A$ generated by $A,\ldots,a^{n-1}(A)$, with the automorphism $a$;
\item the algebra $\mathcal B$ made up of $B$ and its complement, with the identity automorphism $b$.
\end{itemize}
Assume one can jointly embed $(\mathcal A,a)$ and $(\mathcal B,b)$ in $(\mathcal C,c)$; identify $\mathcal A,\mathcal B$ with the subalgebras of $\mathcal C$ associated with these embeddings.
Then $B=B\cap\sqcup_{i=0}^{n-1}c^i(A)=\sqcup_{i=0}^{n-1}c^i(B\cap A)$, so $B$ is cut into $n$ clopen subsets of equal measure. This means that $\frac{r}{n}$ must belong to $V(\mu)$. Hence, the joint embedding property fails for instance when $V= \langle \frac{1}{2}, \frac{1}{ \pi} \rangle$. 

Analysing the above example, one can extract a combinatorial condition on $V$ that is equivalent to the existence of a dense conjugacy class in $H_{\mu_V}$.

\begin{prop}\label{carac_jep}
Let $V$ be a group-like subset of $[0,1]$. Then there is a dense conjugacy class in $H_{\mu_V}$ if, and only if, $V$ satisfies the following condition: whenever $a_i, b_j \in V$ and $n_i,m_j \in \N$ are such that $\sum_{i=1}^p n_i a_i =1= \sum_{j=1}^q m_j b_j$, there exist $c_{i,j} \in V$ such that
$$ \forall j \ m_jb_j=\sum_{i=1}^p \text{lcm}(n_i,m_j) c_{i,j} \quad \text{ and } \quad \forall i \ n_ia_i=\sum_{j=1}^q \text{lcm}(n_i,m_j) c_{i,j} \ .$$

This holds true in particular when $V+ \Z$ is a $\Q$-vector subspace of $\R$, and when $V+ \Z$ is a subring of $\R$.

\end{prop} 

As we already mentioned above, Akin \cite{Akin2005} actually proved that $H_\mu$ has a comeager conjugacy class when $V(\mu)+\Z$ is a $Q$-vector subspace of $\R$, a fact that we will recover below.

\begin{proof}[Proof of Proposition \ref{carac_jep}]
To simplify the notation below we sometimes do not mention the measure; in particular, all automorphisms are to be understood as preserving $\mu$.

Assume that the joint embedding property for partial automorphisms holds, and consider $(a_i,n_i)_{1 \le i \le p}$, $(b_j,m_j)_{1 \le j \le q}$ as above. Then one can consider a finite algebra $\mathcal A$ with clopen atoms $A_{i,k}$ for $k \in\{0,\ldots,n_i-1\}$ such that each $A_{i,k}$ has measure $a_i$, and an automorphism $a$ of $\mathcal A$ such that $a(A_{i,k})=A_{i,k+1}$ for all $i, k$ (where addition is to be understood modulo $n_i$); similarly one can consider a finite algebra $\mcB$ with clopen atoms $B_{j,k}$ ($k \in \{0,\ldots,m_j-1\}$) and the corresponding automorphism $b$ of $\mcB$. For all $i$ we let $A_i= \cup A_{i,k}$ and $B_j= \cup B_{j,k}$.

Then we pick $(\mcC,c)$ such that $(\mcA,a)$ and $(\mcB,b)$ can be embedded in $(\mcC,c)$, where $c$ is an automorphism of the finite algebra $\mcC$, and we identify them with the corresponding subalgebras of $\mcC$. If for some $i,j$ $A_i \cap B_j$ is nonempty, then it is a $c$-invariant clopen set. Any atom of $\mcC$ contained in some $A_{i,k} \cap B_{j,l}$ must have an orbit whose cardinality is a multiple of $\text{lcm}(n_i,m_j)$, so $c_{i,j}=\frac{1}{\text{lcm}(n_i,m_j)} \mu(A_i \cap B_j)$ belongs to $V$.   
Then we have for all $i$:
$$n_i a_i= \mu(A_i)= \sum_{j=1}^q \mu(A_i \cap B_j)= \sum_{j=1}^q c_{i,j} \text{lcm}(n_i,m_j) \ .  $$
The same reasoning holds for $m_j b_j$.

This proves one implication; to prove the converse, let us first note that,
% Actually I don't understand how it reduces to this... Julien: I added (too many?) details. Tomás: thank you, I think it's perfect :)
 given a clopen $U$ and two cycles $a,b$ on $U$ of orders $n,m$ respectively and such that 
$\frac{1}{\text{lcm}(n,m)}\mu(U)$ belongs to $V$, there exists a cycle on $U$ of order $N=\text{lcm}(n,m)$  in which both $a$ and $b$ embed. Such a cycle is obtained by cutting $U$ in $N$ disjoint pieces $C_i$ ($0 \le i \le N-1$) of equal measure, and setting $c(C_i)=C_{i+1}$ (modulo $N$). Then, let $N=nr=ms$; letting $A_0,\ldots, A_{n-1}$ denote the atoms contained in $U$ of the algebra on which $a$ is defined, one obtains the desired embedding by identifying each $A_i$ with $\sqcup_{k=0}^{r-1} C_{nk+i}$, and each $B_j$ with $\sqcup_{k=0}^{s-1} C_{mk+j}$.

Now, let $\alpha, \beta$ in $H_{\mu_V}$ be such that $X= \sqcup_{i=1}^p A_i$, where each $A_i$ is clopen and $\alpha$ is a product of cycles $\alpha_i$ of order $n_i$ on $A_i$, and $X= \sqcup_{j=1}^q B_j$, where each $B_j$ is clopen and $\beta$ is a product of cycles $\beta_j$ of order $m_j$ on $B_j$. From the proof of Theorem \ref{Hrushovski}, we see that it is enough to prove that $\alpha,\beta$ embed in a common element of $H_{\mu_V}$.
 Let $n_i a_i= \mu(A_i)$ and $m_j b_j = \mu(B_j)$, and apply our assumption on $V$ to get $c_{i,j}$ as in the lemma's statement. Let $I$ denote the set of all $(i,j)$ such that $c_{i,j} \ne 0$; we may find a finite subalgebra of $\Clop(X)$ whose atoms $C_{i,j}^k$ ($(i,j) \in I$, $1 \le k \le \text{lcm}(n_i,m_j)$) are of measure $c_{i,j}$. For each $(i,j) \in I$, set
$$D_{i,j}= \bigsqcup_{k=1}^{\text{lcm}(n_i,m_j)} C_{i,j}^k$$
We saw that there exists a cycle $\delta_{i,j}$ on $D_{i,j}$, of order $\text{lcm}(n_i,m_j)$, in which a cycle $\alpha_{i,j}$ on $D_{i,j}$ of order $n_i$ and a cycle $\beta_{i,j}$ on $D_{i,j}$ of order $m_j$ both embed. Let $\delta$ be the product of all $\delta_{i,j}$; $\alpha$ embeds in $\delta$ as the product of all $\alpha_{i,j}$, and $\beta$ embeds in $\delta$ as the product of all $\beta_{i,j}$.

To see that the property under discussion holds true when $V+\Z$ is a subring of $\R$, simply note that in that case
$a_i b_j$ belongs  to $V$; thus $c_{i,j}= a_ib_j \frac{n_im_j}{ \text{lcm}(n_i,m_j)}= a_i b_j \text{gcd}(n_i,m_j)$ works.

When $V+ \Z$ is a $\Q$-vector subspace of $\R$, which is equivalent to saying that $\frac{a}{n} \in V$ for any positive integer $n$ and any $a \in V$, we skip the proof since we will show a stronger property below.
 
\end{proof}
The above criterion is probably of minimal practical interest, since it appears to be fairly hard to check (certainly, it does not help much when tackling the case when $V+ \Z$ is a $\Q$-vector subspace of $\R$).

\begin{defn}
Following Akin \cite{Akin2005}, we say that a group-like subset $V \subseteq [0,1]$ is $\Q$\emph{-like} if $V+\Z$ is a $\Q$-vector subspace of $\R$; this is equivalent to saying that $V$ is group-like and $\frac{1}{n} V\subseteq V$ for any positive integer $n$.
\end{defn}

\begin{prop} If $V$ is $\Q$-like, then $\mathcal K^{\mu_V}_\text{aut}$ satisfies the amalgamation property. (The converse is also true.) Hence $H_{\mu_V}$ has a comeager conjugacy class in that case.
\end{prop}
\begin{proof}

Suppose that $( A,\varphi)$ embeds in $( B,\psi)$ and in $( C,\theta)$. %We see $B$ as a subalgebra of $\mathcal A_V=\text{Flim}\mathcal K$, the algebra on $[0,1]$ generated by the intervals $[0,d]$ for $d\in V$, which is in turn a subalgebra of the standard measure algebra $\mathcal M=\text{MEAS}[0,1]$.
We construct the \emph{boolean} amalgam $( B\otimes_A C,\psi\otimes\theta)$ of $( B,\psi)$ and $( C,\theta)$ over $( A,\varphi)$ in the standard way (see for example \cite{Kechris2012a}), and only need to define the measures. We give an argument in the fashion of the one contained in Theorem 2.1 of \cite{Kechris2012a}.

Fix an atom $a\in A$, and list the atoms of $B$ and $C$ contained in $a$ by $\{b^k_i\}^{k<n}_{i<n_k}$ and $\{c^l_j\}^{l<m}_{j<m_l}$ respectively, where $b^k_i$ and $b^{k'}_{i'}$ are in the same $\psi$-orbit iff $k=k'$, and analogously for the $c^l_j$. We want to define the values $x^{kl}_{ij}=\mu(b^k_i\otimes c^l_j)$. Then we would translate these values in the obvious manner to the products of the atoms of $B$ and $C$ contained in the $\varphi$-translates of $a$; finally, we would proceed analogously for the other orbits of $( A,\varphi)$.

Other than being in $V$, the values $x^{kl}_{ij}$ have to satisfy:
$$0\leq x^{kl}_{ij},$$
$$x^{kl}_{ij}=x^{kl}_{i'j'},$$
$$\sum_{ki}x^{kl}_{ij}=\mu(c^l_j),\ \sum_{lj}x^{kl}_{ij}=\mu(b^k_i).$$
Denoting $x^{kl}=x^{kl}_{ij}$, we can reformulate the conditions as:
$$\sum_k n_km_l x^{kl}=m_l\mu(c^l_0),\ \sum_l n_km_l x^{kl}=n_k\mu(b^k_0).$$
Considered as a system in the variables $y^{kl}=n_km_l x^{kl}$, we can find a solution in $\R$, namely $y^{kl}=n_km_l\frac{\mu(b^k_0)\mu(c^l_0)}{\mu(a)}$. Since $V$ is group-like and dense, there must also be solutions $y^{kl}$ in $V$. Since it is also $\Q$-like, we can take $x^{kl}=\frac{y^{kl}}{n_km_l}$ and we are done.
\end{proof}

The amalgamation property for $\mcK^\mu_{\text{aut}}$ is stronger than the existence of a comeager conjugacy class in $H_\mu$; for instance, if $V(\mu)$ is the set of dyadic numbers, then it follows from \cite{Kechris2007}*{discussion after the statement of Theorem 6.5} that $H_\mu$ has a comeager conjugacy class, but it is easy to see that $\mcK^\mu_{\text{aut}}$ does not have the amalgamation property in that case. It does however admit a cofinal class which satisfies the amalgamation property, which is sufficient to obtain (WAP) (that class is made up of finite subalgebras all of whose atoms have the same measure). \emph{A priori}, the cofinal amalgamation property for $\mcK^\mu_{\text{aut}}$ is itself stronger than (WAP); yet we do not know of an example of measure for which $\mcK^\mu_1$ has (WAP) but $\mcK^\mu_{\text{aut}}$ does not have the cofinal amalgamation property.

\bibliography{mybiblio}

\end{document}